\documentclass{amsart} 
\usepackage{graphicx}

\usepackage{amsmath,amssymb, amsbsy}
\usepackage{enumerate}
\usepackage{hyperref}
\usepackage{tikz}
\usepackage{cite}

\newcommand{\R}{{\mathbb R}}
\newcommand{\N}{{\mathbb N}}
\newcommand{\e }{\varepsilon}
\newcommand{\Em}{\beta_m}
\newcommand{\Emo}{\overline\beta_m}
\newcommand{\primo}{\gamma}
\newcommand{\secondo}{\widetilde \gamma}
\newcommand{\primop}{\phi}
\newcommand{\secondop}{\widetilde\phi}
\newcommand{\F}{\mathcal{F}}
\newcommand{\GG}{\mathcal{G}}
\newcommand{\D}{\mathcal{D}}
\newcommand{\eps}{\varepsilon}
\newcommand{\smx}{\sin(mx)}
\newcommand{\serie}{\sum_{m=1}^\infty}
\newcommand{\ipi}{\int_0^\pi}
\newcommand{\il}{\int_{-\ell}^\ell}

\newcommand{\hs}{H^2_*(\Omega)}
\newcommand{\ho}{H^2_\mathcal{O}(\Omega)}
\newcommand{\he}{H^2_\mathcal{E}(\Omega)}
\newcommand{\hh}{{H_*^{-2}}}
\newcommand{\hho}{{H}^{-2}_\mathcal{O}(\Omega)}
\newcommand{\hhe}{{H}^{-2}_\mathcal{E}(\Omega)}
\newcommand{\Po}{\mathcal{P}_\mathcal{O}}
\newcommand{\Pe}{\mathcal{P}_\mathcal{E}}

\renewcommand{\ge }{\geqslant}
\renewcommand{\geq }{\geqslant}
\renewcommand{\le }{\leqslant}
\renewcommand{\leq }{\leqslant}

\def\neweq#1{\begin{equation}\label{#1}}
\def\endeq{\end{equation}}
\def\eq#1{(\ref{#1})}

\newtheorem{theorem}{Theorem}[section]
 \newtheorem{lemma}{Lemma}[section]
 \newtheorem{proposition}{Proposition}[section]

\newtheorem{definition}{Definition}[section]
\newtheorem{conjecture}{Conjecture}[section]

\title[A Minimaxmax Problem]{A Minimaxmax Problem for Improving the Torsional Stability of Rectangular Plates}

\author{Elvise Berchio}
\thanks{Elvise Berchio (corresponding author),
             Politecnico di Torino, 
             Torino, Italy.
             elvise.berchio@polito.it}
\author{Davide Buoso}
\thanks{Davide Buoso,
            Universidade de Lisboa,
           Lisboa, Portugal.
           dbuoso@fc.ul.pt.
           }
\author{Filippo Gazzola}
\thanks{Filippo Gazzola, 
           Politecnico di Milano 
           Milano, Italy.
           filippo.gazzola@polimi.it.
           }
\author{Davide Zucco}
\thanks{Davide Zucco, 
           Universit\`a di Torino,
           Torino, Italy.
           davide.zucco@unito.it.
}

\keywords{Shape optimization; worst-case optimization; torsional instability; plates; bridges.}
\subjclass[2010]{35J40; 35P15; 74K20.}

\begin{document}

\begin{abstract}
We use a {\em gap function} in order to compare the torsional performances of different reinforced plates under the action of external forces.
Then, we address a shape optimization problem, whose target is to minimize the torsional displacements of the plate: this leads us to set up
a {\em minimaxmax} problem, which includes a new kind of worst-case optimization. Two kinds of reinforcements are considered: one aims at strengthening
the plate, the other aims at weakening the action of the external forces. For both of them, we study the existence of optima within suitable classes
of external forces and reinforcements. Our results are complemented with numerical experiments and with a number of open problems and conjectures.
\end{abstract}

\maketitle

\section{Introduction}

When pedestrians cross a footbridge or the wind hits a suspension bridge, the deck undergoes oscillations, which can be of three different kinds.
The longitudinal oscillations, in the direction of the bridge, are usually harmless because bridges are planned to withstand them. The lateral oscillations, which move the deck horizontally away from its axis, may become dangerous, if the pedestrians walk synchronously; see the recent events at the London Millennium Bridge
\cite{strogatz,macdonald,sanderson} and also earlier dramatic historical events \cite[\S 1.2]{book}. The torsional oscillations, which appear when the deck rotates around its main axis, are the most dangerous ones; they also appear in heavier structures such as suspension bridges; see
\cite[\S 1.3,1.4]{book} for a survey.\par
Following \cite{fergaz}, we view a bridge as a long narrow rectangular thin plate hinged at two opposite edges and free on the remaining
two edges: this plate well describes decks of footbridges and suspension bridges, which, at the short edges, are supported by the ground.
The corresponding Euler-Lagrange equation is given by a fourth order equation in a planar domain complemented by suitable boundary conditions; see \eqref{loadpb}. The solution of this equation represents the
vertical displacement of the plate under the action of an external force. Note that the solution is continuous because for planar domains the energy space embeds into continuous functions. This does not occur in higher space dimensions or for lower order
problems. The continuity of the solution is a crucial feature since it enables us to use the so-called \emph{gap function}
introduced in \cite{bebuga2}. The gap function measures the difference of the vertical displacements on the two free edges of the plate and
is therefore a measure of its torsional response. The number measuring the \emph{maximal gap} is given by the maximum over the free edges of the gap function; see \eqref{funzionale}. Clearly, the maximal gap depends on the force through the Euler-Lagrange equation satisfied by the solution and one is led to seek the force which yields the largest torsional displacement. This gives a measure of the risk that the bridge collapses.
In order to lower this risk, one may try different ways of reinforcing the deck.\par
Imagine that one has a certain amount of stiff material (e.g., steel) and has to decide where to place it within the plate in order to lower the maximal gap and, in turn, the torsional displacements. This material should occupy a proper open subset of the plate. In literature this kind of problem has been tackled in several ways; we refer to \cite{chaken,chaken2,ksw} for related problems on the torsion of a bar. Since this \emph{shape optimization problem} is completely new, we choose
two different strategies: we first assume that the stiff material reinforces a part of the plate, then we assume that it acts directly on the force
and weakens it by a factor involving the characteristic function of the region occupied by the material and a constant measuring the strength of the stiff material. Reinforcing the plate means that we add the stiff material in critical parts of the plate in order to increase the energy necessary to bend it. Weakening the force means that we place some ``aerodynamic damper'' in order to reduce the action of the external force.
These kinds of minimization problems naturally lead to homogenization \cite{murat}, see also \cite{ksw} for a stiffening problem for the torsion of a bar. Homogenization would lead to optimal designs with reinforcements scattered throughout the structure, namely designs impossible to implement for engineers. And since the design of the stiff structure should be usable for engineers, homogenization must be
avoided and the class of admissible geometries for the reinforcements should be sufficiently small. In this respect, we mention the paper by Nazarov-Sweers-Slutskij \cite{nazarov}, where only ``macro'' reinforcements are considered, although in a fairly different setting. The structural optimization problem that we tackle may be seen as the ``dual problem'' of the one considered in the seminal work by Michell \cite{michell}, see also updated results in \cite[Chapter 4]{bendsoe}: our purpose is to determine the best performance of the stiff material by maintaining the cost whereas Michell aimed to determine the cheapest stiff material by maintaining the performance.\par
For both the two mentioned ways of introducing the reinforcement, our purpose is to optimize the maximal gap. We will introduce suitable classes, for both the force and the reinforcement, in which to set up the optimization problem. First we seek the ``worst'' forces for a given reinforcement.
This number yields the maximal gap that may occur. Then, we seek the ``best'' reinforcements, which minimize the effect of the forces.
We are then led to solve a {\em minimaxmax problem}. The existence of a maximal force and of a minimal reinforcement depends on how wide the classes are. In this paper, for the forces we mainly deal with the classes of Lebesgue functions or of the dual of the energy space, while, concerning the reinforcements, we restrict our attention to simple designs, that may be appropriate for engineering applications: cross-type reinforcements, tiles of rectangular shapes, networks of
bounded length, and general Lipschitz domains, see Definition \ref{spazi}.\par
This minimaxmax problem can also be seen as a \emph{worst-case optimization problem}, since one is interested in minimizing the worst value
of a functional among all possible designs.
An extended presentation of worst-case optimization problems in structural mechanics can be found in \cite{alldap}; see also \cite{cher} for a worst-case optimization problem of a compliance functional in the Lebesgue space.
\par
This paper is organized as follows. In Section \ref{2}, we introduce rigorously the gap function with the minimaxmax problem. In Section~\ref{sec3}, we identify suitable classes for which the minimaxmax problem admits a solution (i.e., worst forces with best reinforcements). In Section~\ref{sec4}, we discuss symmetry properties of worst forces in the case of symmetric reinforcements. In Section~\ref{sec5}, we investigate the worst force acting on a plate with no reinforcement. In Section \ref{easy}, we analyze the effects of cross-type reinforcements, while, in Section~\ref{moregen}, we consider more general polygonal-type reinforcements. In both cases we solve numerically the minimaxmax problem. Sections ~\ref{sec8} to \ref{proofdeltas} are dedicated to the proofs of our results. Finally, Section \ref{open} contains the conclusions on the work done.

\section{Variational Setting and Reinforcements for the Plate}\label{2}

Up to scaling, in the following we may assume that the plate $\Omega$ has length $\pi$ and width $2\ell$ with $2\ell\ll\pi$ so that $\Omega=]0,\pi[\times]-\ell,\ell[\subset\R^2$. According to the Kirchhoff-Love theory \cite{Kirchhoff,Love} (see also \cite{mansfield} for a modern presentation), the energy $\mathbb E$ of a
vertical deformation $u$ of the plate $\Omega$ subject to a load $f$ may be computed through the functional
\neweq{EKL}
\mathbb E(u):=\int_\Omega \left(\frac{(\Delta u)^2}{2}+(1-\sigma)(u_{xy}^2-u_{xx}u_{yy})-fu\right)\, dxdy \, ,
\endeq
where $\sigma$ is the Poisson ratio and satisfies $0<\sigma<1$. This implies that the quadratic part of the energy $\mathbb E$ is positive. For the partially hinged plate under consideration,
the functional $\mathbb E$ should be minimized in the space
$$H^2_*(\Omega):=\Big\{v\in H^2(\Omega):\, v=0\mbox{ on }\{0,\pi\}\times ]-\ell,\ell[\Big\}\, ;$$
since $\Omega$ is a planar domain, one has the embedding $H^2(\Omega)\subset C^0(\overline{\Omega})$, and the condition on $\{0,\pi\}\times ]-\ell,\ell[$
is satisfied pointwise. By \cite[Lemma 4.1]{fergaz} we know that $H^2_*(\Omega)$ is a Hilbert space when endowed with the scalar product
$$
(u,v)_{H^2_*}:=\int_\Omega \left[\Delta u\Delta v+(1-\sigma)(2u_{xy}v_{xy}-u_{xx}v_{yy}-u_{yy}v_{xx})\right]\, dxdy\,
$$
and associated norm $\|u\|_{H^2_*}^2=(u,u)_{H^2_*}$, which is equivalent to the usual norm in $H^2(\Omega)$, that is,
$\|u\|_{H^2}^2=\|u\|_{L^2}^2+\|D^2 u\|_{L^2}^2$.
We also define $\hh(\Omega)$ as the dual space of $H^2_*(\Omega)$
and we denote by $\langle\cdot,\cdot\rangle$ the corresponding duality.
If $f\in L^1(\Omega)$ then the functional $\mathbb E$ is well-defined in $H^2_*(\Omega)$, while if $f\in\hh(\Omega)$ we need to replace
$\int_\Omega fu$ with $\langle f,u\rangle$.\par
Assume that the plate $\Omega$ is reinforced with a
stiff material which occupies an open region $D\subset\Omega$ and that $D$ belongs to a certain class $\D$, while $f$ belongs to some space $\F$ of admissible forcing terms. We consider two possible ways of reinforcing the plate: either we stiffen the plate by increasing the cost of the bending energy, or we add an aerodynamic damper by weakening the force. This modifies the original energy \eqref{EKL} into the two following ways:
\neweq{energy-gsD}
\mathbb{E}_1(u):=\int_\Omega \left[(1+d\chi_D)\left(\frac{(\Delta u)^2}{2}+(1-\sigma)(u_{xy}^2-u_{xx}u_{yy})\right)-fu\right]\, dxdy
\endeq
and
\neweq{energy-f}
\mathbb{E}_2(u):=\int_\Omega \left[\frac{(\Delta u)^2}{2}+(1-\sigma)(u_{xy}^2-u_{xx}u_{yy})-\frac{f\, u}{1+d\chi_D}\right]\, dxdy\, ,
\endeq
where $\chi_D$ is the characteristic function of $D$ and $d>0$ is the strength of the stiffening material.
As for \eqref{EKL}, the quadratic part of the functionals \eqref{energy-gsD} and \eqref{energy-f} are positive and should be minimized on the space $H^2_*(\Omega)$.

When dealing with $\mathbb{E}_1$, for any $D \subset \Omega$ open, we introduce the bilinear form
\neweq{scalarpD}
(u,v)_D:=\int_D\left[\Delta u\Delta v+(1-\sigma)(2u_{xy}v_{xy}-u_{xx}v_{yy}-u_{yy}v_{xx})\right]\, dxdy
\endeq
so that $(u,v)_{\Omega}=(u,v)_{H^2_*}$.
Then, for all $f\in H^{-2}_*(\Omega)$ the minimizer of $\mathbb{E}_1$ satisfies the weak Euler-Lagrange equation
\neweq{weakD}
(u_{f,D},v)_{H^2_*}+d(u_{f,D},v)_D=\langle f,v\rangle\qquad\forall v\in H^2_*(\Omega)\,,
\endeq
which has no strong counterpart due to the lack of regularity of the term $(1+d\chi_D)$ that prevents an integration by parts.

On the other hand, due to the lack of regularity of the term $(1+d\chi_D)$, $\mathbb{E}_2$ is not defined for all $f\in H^{-2}_*(\Omega)$, but it is well-defined for any $f\in L^1(\Omega)$; in this case the minimizer satisfies the equation
$(u_{f,D},v)_{H^2_*}=\int_{\Omega} \tfrac{fv}{1+d\chi_D}\,dxdy$, for all $ v\in H^2_*(\Omega)$,
which may also be written in its strong form:
\neweq{loadpb}
\begin{cases}
(1+d\chi_D)\Delta^2 u=f\,, &  \text{in } \Omega\,, \\
u=u_{xx}=0\,, & \text{on } \{0,\pi\}\times ]-\ell,\ell[\,, \\
u_{yy}+\sigma u_{xx}=u_{yyy}+(2-\sigma)u_{xxy}=0\,, & \text{on } ]0,\pi[\times \{-\ell,\ell\}\,.
\end{cases}
\endeq
Since $0<\sigma<1$, both $\mathbb{E}_1$ and $\mathbb{E}_2$ admit a unique critical point
in $ H^2_*(\Omega)$, their absolute minimum. The minimizer may be different for $\mathbb{E}_1$ and $\mathbb{E}_2$ but we will denote both of them by $u_{f,D}$ since it will always be clear which functional we are dealing with. As we have just seen, the solution
$u_{f,D}$ satisfies a weak Euler-Lagrange equation for $\mathbb{E}_1$ (but not a strong one) while it satisfies a strong Euler-Lagrange equation for $\mathbb{E}_2$ (and not a merely weak one).\par
Assume that some classes $\F$ and $\D$ of admissible $f$ and $D$ are given. Take $f\in\F$, $D\in\D$, and the minimizer $u_{f,D}\in H^2_*(\Omega)\subset C^0(\overline{\Omega})$ of $\mathbb{E}_1$ or $\mathbb{E}_2$, then compute its \emph{gap function} with its \emph{maximal gap}:
\neweq{funzionale}
\GG_{f,D}(x):=u_{f,D}(x,\ell)-u_{f,D}(x,-\ell)\, ,\qquad\GG_{f,D}^\infty:= \max_{x\in[0,\pi]}\, |\GG_{f,D}(x)|\, .
\endeq
In this way we have defined the map $\GG^\infty_{f,D}\colon \F\times\D\to[0,\infty[$ with $(f,D)\mapsto \GG^\infty_{f,D}\,.$
Given $D\in\mathcal D$, we first seek the worst $f\in\F$ such that
\neweq{GGD}
\GG_D^\infty:= \max_{f\in\F}\, \GG_{f,D}^\infty = \max_{f\in\F}\, \max_{x\in[0,\pi]}\, |\GG_{f,D}(x)|\, ,
\endeq
and then the best $D\in\D$ such that
\neweq{minimaxmax2}
\GG^\infty := \min_{D\in\D}\, \GG_D^\infty\,
=\, \min_{D\in\D}\, \max_{f\in\F}\, \max_{x\in[0,\pi]}\, |\GG_{f,D}(x)|\, .
\endeq
This is our \emph{minimaxmax problem}. In the next sections we analyze some classes $\F$ and $\D$, where \eq{GGD} and \eqref{minimaxmax2} admit a solution.
Note that $\GG^\infty=\GG^\infty(\F,\D)$ is monotone with respect to both the classes $\F$ and $\D$ but with opposite monotonicity.

\section{Existence Results for the Minimaxmax Problem}\label{sec3}

We determine some classes $\F$ and $\D$ of admissible forces and reinforcements for which \eq{GGD} and \eqref{minimaxmax2} admit a solution.
The proofs are given in Section \ref{sec8}. We first show that $\GG^\infty_D$, as in \eqref{GGD}, is well-defined for some choices of the class $\F$.

\begin{theorem}\label{optimal}
For a given open set $D\subset\Omega$ and $p\in ]1, +\infty]$, the maximization problems
\begin{equation}\label{functional}
\max \big \{  \GG^\infty_{f,D}:\, \text{$f\in \hh(\Omega)$ with $\|f\|_{\hh}=1$}\big\}\,,\quad (\text{for $\mathbb{E}_1$}),
\end{equation}
\neweq{maxim}
\max \big \{  \GG^\infty_{f,D}:\, \text{$f\in L^p(\Omega)$ with $\|f\|_{L^p}=1$}\big\}\,,\quad (\text{for both $\mathbb{E}_1$ and $\mathbb{E}_2$}),
\endeq
admit a solution (in the considered space).
\end{theorem}

Then, we turn to problem \eqref{minimaxmax2}. We introduce some classes $\D$ for which it is guaranteed the existence of a solution.

\begin{definition}[Classes of admissible reinforcements]\label{spazi}

(a) \emph{Cross-type reinforcements}: for $N,M\in\mathbb N$, $\mu\in ]0,\pi/2N[$, $\varepsilon\in]0,\ell/M[$, $x_i\in [\mu, \pi-\mu]$ for $i=1,\dots, N$ with
$x_{i+1}-x_i>2\mu$ for $i\le N-1$, and $y_j\in[-\ell+\varepsilon,\ell-\varepsilon]$ for $j=1,\dots, M$ with $y_{j+1}-y_j>2\varepsilon$ for $j\le M-1$, define
$$\mathcal C :=\Big\{D\subset\Omega: \, D=\Big(\bigcup_{i=1}^N ]x_i-\mu,x_i+\mu[\times ]-\ell,\ell[\Big)\cup\Big(\bigcup_{j=1}^M ]0,\pi[\times ]y_j-\varepsilon,y_j+\varepsilon[ \Big)\Big\}\,.$$

(b)  \emph{Tiles of rectangular shapes}: for $N\in\mathbb N$ and $\varepsilon\in ]0,\ell[$, define
$$\mathcal T :=\Big\{D\subset\Omega: \, D=\bigcup_{i=1}^N R^i, \text{$R^i\subset\Omega$ is an open rectangle with inradius  $\geq\varepsilon$
}\Big\}\,.$$

(c) \emph{Networks of bounded length}: for $\varepsilon\in]0, \ell[$ and $L>0$, define
$$\mathcal N :=\big\{D\subset\Omega:\, \text{$D=\Sigma^\varepsilon$ where $\Sigma\subset \overline\Omega$ is closed, connected, $\mathcal H^1(\Sigma)\leq L$}\big\}\,.$$
Here $\mathcal H^1$ denotes the one-dimensional Hausdorff measure of a set, and $\Sigma^\varepsilon$ represents the $\varepsilon$-tubular neighborhood of $\Sigma$, namely the set of points in $\Omega$ at distance to $\Sigma$ less than $\varepsilon$.

(d) \emph{Lipschitz trusses}: for $\varepsilon\in ]0,\ell[$, define
$$\mathcal L :=\big\{D\subset \Omega: \, \text{$D$ open with the inner $\varepsilon$-cone property}\big\}\,.$$
We recall that by the inner $\varepsilon$-cone property we mean that at every point $x$ of the boundary $\partial D$ there is some truncated cone from $x$ with an opening angle $\varepsilon$ and radius $\varepsilon$ inside $D$.
\end{definition}

Notice that some of these classes are \emph{monotone} with respect to set inclusion, namely $\mathcal C\subset \mathcal T\subset \mathcal L$,
for suitable choices of the parameters $N, \varepsilon, \mu, L$.

\begin{theorem}\label{optimaltruss}
For a given $\kappa\in ]0,2\pi\ell[$ the minimization problem
\begin{equation}\label{prob}
\min\{ \GG_D^\infty:\, \text{$D\in\D$  with $|D|=\kappa$}\}\,,
\end{equation}
admits a solution whenever the class $\mathcal D$ is one of those introduced in Definition~\ref{spazi} (with the parameters chosen so as to satisfy the area constraint).
\end{theorem}

\section{Symmetric Framework for the Minimaxmax Problem}\label{sec4}

Whenever the class $\D$ of the minimaxmax problem \eqref{minimaxmax2} reduces to symmetric reinforcements, the class $\F$ can be reduced without changing the problem. We say that a set $D\subset\Omega$ is \emph{symmetric with respect to the midline} (or, for short, symmetric) if
$(x,y)\in D$ if and only if $(x,-y)\in D$, for all $(x,y)\in\Omega$.
Then, we introduce the subspaces of even and odd functions with respect to $y$:
$$\begin{array}{cc}
\he:=\{u\in\hs:\, u(x,-y)=u(x,y)\ \forall(x,y)\in\Omega\}\, ,\\
\ho:=\{u\in\hs:\, u(x,-y)=-u(x,y)\ \forall(x,y)\in\Omega\}\, .
\end{array}$$
We first notice that
\neweq{scomp}
\he\perp\ho\, ,\qquad \hs=\he\oplus\ho\, .
\endeq
For all $u\in\hs$ we denote by $u^e\in\he$ and $u^o\in\ho$ its components according to this decomposition, namely
$u^e(x,y)=\tfrac{u(x,y)+u(x,-y)}{2}$ and $u^o(x,y)=\tfrac{u(x,y)-u(x,-y)}{2}$.
The orthogonal projections $
\Pe\colon\hs\to\he$ and $\Po\colon\hs\to\ho$ are defined onto these subspaces as $\Pe u:=u^e$ and $\Po u:=u^o$, for every $u\in\hs$.
Then, we define:
$$\begin{array}{cc}
\hhe:=\{f\in\hh(\Omega): \langle f,v\rangle=0\ \forall v\in\ho\}\,,\\
\hho:=\{f\in\hh(\Omega): \langle f,v\rangle=0\ \forall v\in\he\}\,.
\end{array}$$
In particular, $\ho\subseteq\ker f$ for every $f\in\hhe$ and $\he\subseteq\ker f$ for every $f\in\hho$.
Moreover, $\hh(\Omega)=\hhe\oplus\hho$, that is for every $f\in\hh(\Omega)$ there exists a unique couple $(f^e,f^o)\in\hhe\times\hho$ such that $f=f^e+f^o$; with $f^e:=f\circ\Pe$ and $f^o:=f\circ\Po$.
As usual, we endow $\hh(\Omega)$ with the norm
$\|f\|_{\hh}\, :=\, \sup_{\|v\|_{H^2_*=1}}\ \langle f,v\rangle\, ,$
and we observe that
\neweq{normf}
\|f\|_{\hh}=\max\big\{\|f^e\|_{\hh},\|f^o\|_{\hh}\big\}\qquad\forall f\in\hh(\Omega)\, .
\endeq

The next result shows that if the reinforcement $D$ is symmetric with respect to the midline then the worst forces $f$, whose existence is ensured by Theorem~\ref{optimal}, can be sought in the class of odd distributions.

\begin{theorem}\label{fsymmetry}
Assume that $D\subset\Omega$ is open and symmetric with respect to the $x$-axis. Then, \eqref{functional} is equivalent to
$\max\{\GG^\infty_{f,D}:\, \text{$f\in\hho$, $\|f\|_{\hh}=1$}\}.$
Moreover, if $f\in\hho$ is such that $\|f\|_{\hh}=1$ and
$$\GG^\infty_{f,D}=\max\Big\{\GG^\infty_{g,D}:\,  \text{$g\in\hh(\Omega)$ with $\|g\|_{\hh}=1$}\Big\}\, ,$$
then there exist infinitely many $g\in\hh(\Omega)$ such that $\Pe g\neq0$, $\|g\|_{\hh}=1$, and $\GG^\infty_{g,D}=\GG^\infty_{f,D}$.
\end{theorem}

Theorem \ref{fsymmetry} states that, for a symmetric reinforcement $D$ (possibly $D=\emptyset$ as for the free plate),
the maximization of the gap function can be restricted to the class of odd distributions.
But Theorem \ref{fsymmetry} {\em does not} state that {\em only} odd $f$ attain the maximum. And indeed, $\GG^\infty_{f,D}$ is not sensitive to the
addition of some $\phi\in\hhe$ to $f$, provided that the total norm is not exceeded. An interesting open problem is to determine whether there
exists a unique $f\in\hho$ maximizing $\GG^\infty_{f,D}$ (up to a sign change). We expect the answer to depend on $D$, in particular on possible
additional symmetry properties of $D$. We prove Theorem \ref{fsymmetry} in Section \ref{pp}. \par
Next, we have the following $L^p$-version of Theorem \ref{fsymmetry}.

\begin{theorem}\label{fsymmetry2}
Assume that $D\subset\Omega$ is open and symmetric with respect to the $x$-axis and let $p\in ]1,\infty]$. Then, problem \eqref{maxim} is equivalent to the maximization problem
$\max\Big\{\GG^\infty_{f,D}:\, \text{$f\in L^p(\Omega)$, $f$ odd in $y$, $\|f\|_{L^p}=1$}\Big\}$.

If $1<p<\infty$, then any maximizer is necessarily odd with respect to $y$.\par
If $p=\infty$ and an odd maximizer $f$ satisfies $|f(x,y)|<1$ on a subset of $\Omega$ of positive measure, then there exist infinitely many
maximizers $g\in L^\infty(\Omega)$ such that $g^e\neq0$ and $\|g\|_{L^{\infty}}=1$.
\end{theorem}

Theorem \ref{fsymmetry2} states that, for a symmetric reinforcement $D$, the maximization of the gap function can be restricted to the class of odd functions.
Moreover, differently from Theorem \ref{fsymmetry}, if $1<p<\infty$ it says that {\em only} odd functions $f$ attain the maximum. On the other hand, in the case $p=\infty$ oddness may fail, provided that there exists an odd maximizer
satisfying the somewhat strange property stated in Theorem \ref{fsymmetry2}: the reason of this assumption will become clear in the proof of Theorem \ref{fsymmetry2} given in Section \ref{pp}.

\section{Worst Cases on the Free Plate}\label{sec5}

In this section, we consider the free plate with no reinforcement ($D=\emptyset$) so that $\mathbb{E}_1$ and $\mathbb{E}_2$ coincide, and we study problem \eqref{functional}. For simplicity,
for all $f\in H^{-2}_*(\Omega)$, we set $\GG_{f}(x)= \GG_{f,\emptyset}(x)$ and $ \GG^\infty_{f}= \GG^\infty_{f,\emptyset}$.\par
Following the suggestion of Theorem \ref{fsymmetry}, for any $z\in]0,\pi[$ we focus on the \emph{odd} distribution
\begin{equation}\label{defTi}
T_{z}:=\frac{\delta_{(z,\ell)}-\delta_{(z,-\ell)}}{2}\in \hho\,,
\end{equation}
where $\delta_{P}$ is the Dirac delta with mass concentrated at $P\in \overline\Omega$.
Let $u_z\in H^2_*(\Omega)$ be the unique solution of the equation $(u_z,v)_{H^2_*}=\langle T_{z},v\rangle$, for all $v\in H^2_*(\Omega)$.
By the Riesz Theorem, this means that $u_z$ is the representative of $T_{z}$ and therefore (by taking $v=u_z$)
$\|T_{z}\|_{H^{-2}_*}^2=\|u_z\|_{H^2_*}^2=\langle T_{z},u_z\rangle={\GG_{T_{z}}(z)}/{2}$.
This enables us to normalize $T_{z}$ and introduce the distribution $\overline T_{z}:=\tfrac{\sqrt{2}T_{z}}{\sqrt{{\GG_{T_{z}}(z)}}}$ such that $\|\overline T_{z}\|_{H^{-2}_*}=1$.
For any integer $m$, set
\neweq{Um}
\Upsilon_m:=\frac{\sinh^2(m\ell)}{m^3\left[(3+\sigma)\sinh(m\ell)\cosh(m\ell)+(1-\sigma)m\ell\right]}\, .
\endeq
In Section \ref{proofdeltas} we prove the following result.

\begin{proposition}\label{GTxi}
For all $x,z\in]0,\pi[$ we have
$$\GG_{T_{z}}(x)= \frac{4}{\pi(1-\sigma)}\serie \Upsilon_m\, \sin(mz)\, \sin(m x)\, ,
\qquad \GG_{\overline T_{z}}(x)=\frac{\sqrt{2}\ \GG_{ T_{z}}(x)}{\sqrt{\GG_{T_{z}}(z)}}\, .$$
\end{proposition}

Let us explain how Proposition \ref{GTxi} suggests a conjecture for the solution of \eqref{functional} when $D=\emptyset$.
Let $\Upsilon_m$ be as in \eqref{Um} and consider the function $\Phi$ defined as
$\Phi(x):=\serie \Upsilon_m\sin^2(m x)$, for every $x\in[0,\pi]$.
Note that $\Phi(x)>0$ for all $x\in]0,\pi[$ and
\neweq{propPhi}
\Phi(0)=\Phi(\pi)=0\, ,\quad\Phi(\tfrac\pi2 )=\sum_{k=0}^\infty\Upsilon_{2k+1}\, ,\quad\Phi'(\tfrac\pi2 )=0\, ,\quad\Phi''(\tfrac\pi2 )<0\, .
\endeq
Some numerical computations and \eqref{propPhi} suggest that $\Phi$ achieves its maximum at $x=\pi/2$:
\neweq{crucial}
\Phi(\tfrac\pi2 )> \Phi(x)\qquad\forall x\neq \tfrac\pi2\, .
\endeq

Moreover, by H\"older's inequality, Proposition \ref{GTxi}, and condition \eqref{crucial}, for every $x,z\in]0,\pi[$
\begin{eqnarray*}
\left|\GG_{\overline T_{z}}(x)\right| &=& \frac{4\, \sqrt{2}}{\pi(1-\sigma)\, \sqrt{\GG_{T_{z}}(z)}}\
\left|\serie \Upsilon_m\, \sin(mz)\, \sin(m x)\right|\\
\ &\le& \frac{4\, \sqrt{2}}{\pi(1-\sigma)\, \sqrt{\GG_{T_{z}}(z)}}\
\serie \sqrt{\Upsilon_m}\, |\sin(mz)|\, \sqrt{\Upsilon_m}\, |\sin(m x)|\\
 &\le& \frac{4\, \sqrt{2}}{\pi(1-\sigma)\, \sqrt{\GG_{T_{z}}(z)}}\
\bigg(\serie \Upsilon_m\, \sin^2(mz)\bigg)^{\tfrac12}\cdot\bigg(\serie \Upsilon_m\, \sin^2(mx)\bigg)^{\tfrac12}\\
 &=& \frac{2 \sqrt{2}}{\sqrt{\pi(1-\sigma)}}\, \Phi(x)^{\tfrac12}\le \frac{2 \sqrt{2}}{\sqrt{\pi(1-\sigma)}}\, \Phi(\tfrac\pi2 )^{\tfrac12}\, .
\end{eqnarray*}
Note that the above application of the H\"older inequality yields a strict inequality whenever $z\neq x$.
Therefore, after taking the maximum over $[0,\pi]$ we deduce that
$\GG_{\overline T_{z}}^\infty< \tfrac{2 \sqrt{2}}{\sqrt{\pi(1-\sigma)}}\, \Phi(\tfrac\pi2 )^{1/2}$ for every $z\neq \tfrac\pi2$ and that for $z= \tfrac\pi2$ the equality holds
$\GG_{\overline T_{\pi/2}}^\infty=\GG_{\overline T_{\pi/2}}(\tfrac\pi2 )=\frac{2 \sqrt{2}}{\sqrt{\pi(1-\sigma)}}\, \Phi(\tfrac\pi2 )^{1/2}$.
Hence, if \eqref{crucial} holds, then we would infer that
\begin{center}
{\em for all $z\in]0,\pi[$ we have $\GG_{\overline T_{z}}^\infty\le\GG_{\overline T_{\pi/2}}^\infty$ with equality if and only if $z=\pi/2$}.
\end{center}
This statement would prove that among all concentrated loads on the free edges of the plate $\Omega$, the largest maximal gap is obtained when the load is concentrated (with opposite signs) at the midpoints $(\pi/2,\pm\ell)$. A numerical support of this fact is provided by Table~\ref{tabella1bis} below. The values collected there have been obtained using the software Mathematica, approximating the Fourier series for $\GG_{\overline T_{z}}^\infty$ up to the 10,000-th term.
\begin{table}[h]
\caption{Numerical values of $10^4\times \GG^{\infty}_{\overline T_z}$ and $10^4\times \GG^{\infty}_{T_z}$ (with $\ell={\pi}/{150}$ and $\sigma=0.2$).}\label{tabella1bis}
\vspace{.1 in}
\begin{center}
\resizebox{\columnwidth}{!}{
\begin{tabular}{|c|c|c|c|c|c|c|c|c|c|c|c|c|c|}
\hline
$z$\!&\!$\frac{\pi}{20}$\!&\!$\frac{\pi}{18}$\!&\!$\frac{\pi}{16}$\!&\!$\frac{\pi}{14}$\!&\!$\frac{\pi}{12}$\!&\!$\frac{\pi}{10}$\!&\!$\frac{\pi}{8}$
\!&\!$\frac{\pi}{6}$\!&\!$\frac{\pi}{4}$\!&\!$\frac{\pi}{2}$\\
\hline
$10^4\times\GG^{\infty}_{\overline T_z}$\!&\!$627.809$\!&\!$659.067$\!&\!$695.691$\!&\!$739.38$\!&\!$792.677$\!&\!$859.592$\!&\!$946.815$\!&\!$1066.21$
\!&\!$1238.29$\!&\!$1429.87$\\
\hline
$10^4\times \GG^{\infty}_{T_z} $\!&\!$19.326$\!&\!$21.354$\!&\!$23.854$\!&\!$27.012$\!&\!$31.123$\!&\!$36.686$\!&\!$44.609$\!&\!$56.687$\!&\!$76.596$\!&\!$102.23$\\
\hline
\end{tabular}}
\end{center}
\end{table}

It is evident that the worst case is attained for $z=\tfrac\pi 2$ and that the map $z\mapsto \GG_{\overline T_{z}}^\infty$ is
increasing on $[0,\pi/2]$ (note that it is symmetric with respect to $\pi/2$). For later use, we put in Table \ref{tabella1bis}
also the values of $ \GG^{\infty}_{T_z}$.

\section{Weakening the Force with Cross-Type Reinforcements}\label{easy}

In this section, we minimize the energy $\mathbb{E}_2$ given in \eqref{energy-f} finding the explicit solution and, in turn, the explicit gap function for particular choices of forces $f$ and reinforcements $D$.
We take symmetric cross-type reinforcements $D\in \mathcal C$ (see Definition \ref{spazi}) with one horizontal arm and $2N+1$ vertical arms for some non-negative integer $N$. More precisely,  fix
$0<\mu<\tfrac{(2N+1)\pi}{4(N+1)}$, $0<\varepsilon<\ell$,
(where the first condition prevents overlapping of vertical arms) and consider the set
\neweq{Deps}
\small
D_{\varepsilon,\mu}^{N}:=\big(]0,\pi[\times ]-\varepsilon,\varepsilon[ \big) \bigcup_{ i=1}^{2N+1}
\left(\left(\tfrac{\pi i}{2N+2}-\tfrac{\mu}{2N+1},\tfrac{\pi i}{2N+2}+\tfrac{\mu}{2N+1}\right)\times ]-\ell,\ell[\right)\,.
\endeq
We will drop the subscripts in $D^N_{\varepsilon, \mu}$ in order to lighten the notation, writing them when needed to avoid confusion.
To compare the effect of the reinforcements on the torsional instability, we are keeping the area of the set $D^{N}$ fixed, indeed we have
$|D^N|=2\pi\varepsilon+4\mu(\ell-\varepsilon)$ for any $N$. Furthermore, for $g\in L^2(]0,\pi[)$ and $\alpha > 0$ with $\alpha\not\in\mathbb{N}$
(since this simplifies some computations), following the suggestion of Theorem \ref{fsymmetry2}, we consider the odd function
\neweq{f2sinh}
f_\alpha(x,y):=R_\alpha \sinh(\alpha y)g(x)
\end{equation}
with $R_\alpha:=\tfrac{\alpha}{2C_g(\cosh(\alpha \ell)-1)}$ and  $C_g:=\int_0^{\pi}|g(x)|\,dx$, so that $\|f_\alpha\|_{L^1}=1$.
We define
\begin{equation}\label{Emd}
\Emo:=\tfrac{2\, \primo_m \, \Upsilon_m}{C_g(1-\sigma)}   \text{ and } \overline \omega_m:=\tfrac{\primo_m}{C_g}\,\tfrac{(1+\sigma)\sinh (m\ell) \cosh (m\ell) +(1-\sigma)m\ell}{(1\!-\!\sigma)m^2[(3\!+\!\sigma)\sinh(m\ell)\cosh(m\ell)+(1\!-\!\sigma)m\ell]}\,,
\end{equation}
where the coefficients $\Upsilon_m$ are as defined in \eq{Um}, and
\neweq{gtilde}
\small
\primo_m:=\frac{2}{\pi}\ipi g(x)\smx\, dx-\tfrac{2d}{\pi(1+d)} \sum_{ i=1}^{2N+1}\int_{\tfrac{\pi i}{2N+2}-\tfrac{\mu}{2N+1}}^{\tfrac{\pi i}{2N+2}+\tfrac{\mu}{2N+1}} g(x)\smx\, dx\,.
\endeq

Then, we obtain an explicit form for the gap function corresponding to problem \eqref{loadpb} with $f=f_\alpha$ and $D=D^N$, and we analyze its asymptotic behavior as $\alpha\to +\infty$.

\begin{theorem}\label{gap_limit_swiss}
Let $\alpha>0$ with $\alpha\not\in\mathbb{N}$, let $u_\alpha$ be the unique solution of \eqref{loadpb} with $f=f_\alpha$ and $D=D^N$, let $\GG_\alpha$ be as in \eqref{funzionale}
with $u_{f,D}=u_\alpha$. As $\alpha\to+\infty$, $\GG_\alpha(x)= \serie \Em(\alpha) \smx$ converges uniformly on $[0,\pi]$ to the function $\overline \GG(x):=  \serie \Emo \smx$, where the Fourier coefficients $\Em(\alpha)$ are so that
$\Em(\alpha)=\Emo-\tfrac{  \overline \omega_m}{\alpha }+o\left(\tfrac{1}{\alpha}\right)$,
with $\Emo$ and  $\overline \omega_m>0$ given in \eqref{Emd}.
\end{theorem}

In Section \ref{expl sol} we prove Theorem \ref{gap_limit_swiss}. We derive the explicit value of $\Em(\alpha)$ in \eqref{betam}. Furthermore, we show that $\overline \GG(x)$ is the gap function corresponding to a solution of the limit problem \eqref{weaklimit2}.

We exploit Theorem \ref{gap_limit_swiss} to numerically solve the minimaxmax problem \eq{minimaxmax2}. More precisely, we fix $\ell=\pi/150$ and $\sigma=0.2$ (two reasonable values for plates modeling the deck of a bridge, see \cite{befega}). Moreover,
we take $f_\alpha$ in \eq{f2sinh} with $g(x)=\sin (n x)$ for $n=1,\dots,10$ and we call $f^n$ its $\hh(\Omega)$ limit as $\alpha \to +\infty$ (see Lemma \ref{limitfalpha_rinforzo}) and $\overline\GG_n$ the corresponding gap function. Then, we consider
\begin{equation}\label{FDn}
{\mathcal F}=\{f^1,...,f^{10}\}\quad \text{and} \quad  {\mathcal D}=\{D^0,...,D^{5}\}\,.
\end{equation}
The results are summarized in Table \ref{tabella1}, in terms of the maximal gap $\overline\GG^{\infty}_n$.
The numerical values in Table \ref{tabella1} have been obtained using the software Mathematica, approximating the Fourier series for $\overline \GG_n$ up
to the 250-th term.

\begin{table}[h]
\begin{center}
\caption{Numerical values of $10^4\times\overline\GG^{\infty}_n$, with $\ell={\pi}/{150}$,  $\sigma=0.2$, $d=2$, $g(x)=\sin(nx)$, $D=D^{N}$, $\mu=0.3$ (above) and $\mu=0.5$ (below).}\label{tabella1}
\vspace{.1 in}
\resizebox{\columnwidth}{!}{
\begin{tabular}{|c|c|c|c|c|c|c|c|c|c|c|c|c|c|}
\hline
$10^4\times$\!&\!$\overline\GG^{\infty}_1$\!&\!$\overline\GG^{\infty}_2$\!&\!$\overline\GG^{\infty}_3$\!&\!$\overline\GG^{\infty}_4$\!&\!$\overline\GG^{\infty}_5$\!&\!$\overline\GG^{\infty}_6$\!&\!$\overline\GG^{\infty}_7$\!&\!$\overline\GG^{\infty}_8$\!&\!$\overline\GG^{\infty}_9$\!&\!$\overline\GG^{\infty}_{10}$
\\
\hline
$\emptyset$ \!&\!$65.444$\!&\!$16.357$\!&\!$7.2665$\!&\!$4.0849$\!&\!$2.6123$\!&\!$1.8123$\!&\!$1.3300$\!&\!$1.0170$\!&\!$0.8023$\!&\!$0.6488$\\
\hline
$D^0$ \!&\!$47.113$\!&\!$15.980$\!&\!$14.249$\!&\!$4.4296$\!&\!$11.422$\!&\!$2.6591$\!&\!$7.5961$\!&\!$2.0675$\!&\!$3.3673$\!&\!$1.6048$\\
\hline
$D^1$ \!&\!$53.964$\!&\!$13.158$\!&\!$6.3585$\!&\!$4.0133$\!&\!$3.1797$\!&\!$2.9515$\!&\!$10.284$\!&\!$1.0582$\!&\!$9.9445$\!&\!$2.5730$\\
\hline
$D^2$ \!&\!$55.292$\!&\!$13.979$\!&\!$5.9848$\!&\!$3.4987$\!&\!$2.1837$\!&\!$1.8092$\!&\!$1.4864$\!&\!$1.3153$\!&\!$1.2857$\!&\!$2.8377$\\
\hline
$D^3$ \!&\!$55.839$\!&\!$13.892$\!&\!$6.2568$\!&\!$3.3920$\!&\!$2.2970$\!&\!$1.6152$\!&\!$1.2488$\!&\!$1.0158$\!&\!$0.8667$\!&\!$0.7611$\\
\hline
$D^4$ \!&\!$56.135$\!&\!$14.080$\!&\!$6.2664$\!&\!$3.5181$\!&\!$2.1798$\!&\!$1.6029$\!&\!$1.1965$\!&\!$0.9264$\!&\!$0.7631$\!&\!$0.6461$\\
\hline
$D^{5}$\!&\!$56.320$\!&\!$14.050$\!&\!$6.2225$\!&\!$3.5437$\!&\!$2.2726$\!&\!$1.5216$\!&\!$1.1736$\!&\!$0.8864$\!&\!$0.7190$\!&\!$0.5998$\\
\hline
\end{tabular}}
\end{center}
\par
\begin{center}
\resizebox{\columnwidth}{!}{
\begin{tabular}{|c|c|c|c|c|c|c|c|c|c|c|c|c|c|}
\hline
$10^4\times$\!&\!$\overline\GG^{\infty}_1$\!&\!$\overline\GG^{\infty}_2$\!&\!$\overline\GG^{\infty}_3$\!&\!$\overline\GG^{\infty}_4$\!&\!$\overline\GG^{\infty}_5$\!&\!$\overline\GG^{\infty}_6$\!&\!$\overline\GG^{\infty}_7$\!&\!$\overline\GG^{\infty}_8$\!&\!$\overline\GG^{\infty}_9$\!&\!$\overline\GG^{\infty}_{10}$
\\
\hline
$\emptyset$ \!&\!$65.444$\!&\!$16.357$\!&\!$7.2665$\!&\!$4.0849$\!&\!$2.6123$\!&\!$1.8123$\!&\!$1.3300$\!&\!$1.0170$\!&\!$0.8023$\!&\!$0.6488$\\
\hline
$D^0$ \!&\!$37.707$\!&\!$14.748$\!&\!$16.541$\!&\!$5.3424$\!&\!$7.7463$\!&\!$3.7421$\!&\!$2.8388$\!&\!$1.8136$\!&\!$5.1181$\!&\!$1.1708$\\
\hline
$D^1$ \!&\!$46.544$\!&\!$11.277$\!&\!$5.8180$\!&\!$4.0207$\!&\!$3.5078$\!&\!$3.6263$\!&\!$13.807$\!&\!$1.1909$\!&\!$12.949$\!&\!$3.1775$\\
\hline
$D^2$ \!&\!$48.602$\!&\!$12.383$\!&\!$5.2331$\!&\!$3.0977$\!&\!$2.2697$\!&\!$1.7983$\!&\!$1.5803$\!&\!$1.4839$\!&\!$1.5789$\!&\!$3.7815$\\
\hline
$D^3$ \!&\!$49.473$\!&\!$12.287$\!&\!$5.5878$\!&\!$2.9641$\!&\!$2.0589$\!&\!$1.4918$\!&\!$1.1993$\!&\!$1.0118$\!&\!$0.9058$\!&\!$0.8400$\\
\hline
$D^4$ \!&\!$49.950$\!&\!$12.559$\!&\!$5.6012$\!&\!$3.1382$\!&\!$1.9056$\!&\!$1.4510$\!&\!$1.1105$\!&\!$0.8695$\!&\!$0.7391$\!&\!$0.6464$\\
\hline
$D^{5}$\!&\!$50.251$\!&\!$12.526$\!&\!$5.5421$\!&\!$3.1782$\!&\!$2.0384$\!&\!$1.3462$\!&\!$1.0587$\!&\!$0.8055$\!&\!$0.6678$\!&\!$0.5581$\\
\hline
\end{tabular}}
\end{center}
\end{table}

Several comments are in order. First we notice that, as expected from the statement of Theorem \ref{gap_limit_swiss}, the results do not depend on $\eps$. Moreover, $\mu=0.3$ means that the free edges of the plate are covered by the reinforcement on a percentage of $19\%$ of their length, whereas $\mu=0.5$ means that such a percentage is $31.8\%$. It is worth noting that there is no monotonicity of $\overline\GG^\infty_n$ with respect to either the number of branches, or to the
frequency of $\sin(nx)$, nor to the reinforcement thickness $\mu$.
Also, we observe that each forcing term has its own ``best truss'' yielding a minimal maximal gap: the pattern is quite clear
and it follows a descending diagonal in the two Tables \ref{tabella1}. Basically, we see that $D^{n-1}$ (i.e., the cross with $2n-1$ vertical arms) is the ``best truss'' for $g(x)=\sin(nx)$ and the reason is that the plate is reinforced in the points where $g$ attains either a maximum or a minimum;
we did not display all the related lines but the same pattern holds true until $N=10$.
In particular, if $n=2$ we know that $D^1$
is the best reinforcement since there are parts of the truss under the two extremal points of $g(x)=\sin(2x)$, see the left-hand picture in Figure \ref{seni}
where we depict the longitudinal behavior of $g(x)=\sin(2x)$ and the truss $D^1$ (black spots on the horizontal axis).
\begin{figure}[h]
\begin{center}
\begin{tikzpicture}[scale=0.8, domain=0:3.14]
  \draw[->] (-0.2,0) -- (3.5,0) node[below] {\small $x$};
  \draw[->] (0,-1.2) -- (0,1.5) node[left] {\small $y$};
  \draw[thick]   plot[smooth] (\x,{sin(2*\x r)}) ;
  \fill  (0.8,0) ellipse (2pt and 1pt);
  \fill  (1.6,0) ellipse (2pt and 1pt);
  \fill  (2.4,0) ellipse (2pt and 1pt);
  \end{tikzpicture}
\hspace{5mm}
\begin{tikzpicture}[scale=0.8, domain=0:3.14]
  \draw[->] (-0.2,0) -- (3.5,0) node[below] {\small $x$};
  \draw[->] (0,-1.2) -- (0,1.5) node[left] {\small $y$};
  \draw[thick]   plot[smooth] (\x,{sin(7*\x r)}) ;
  \fill  (0.8,0) ellipse (2pt and 1pt);
  \fill  (1.6,0) ellipse (2pt and 1pt);
  \fill  (2.4,0) ellipse (2pt and 1pt);
\end{tikzpicture}
\caption{The forces $g(x)=\sin(2x)$ (left) and $g(x)=\sin(7x)$ (right) with the truss $D^1$.}\label{seni}
\end{center}
\end{figure}
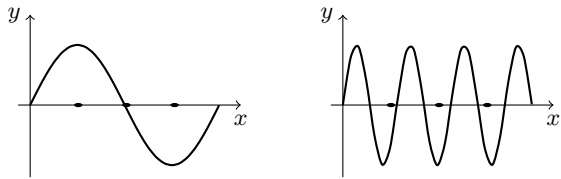

We remark that some trusses aggravate the torsional instability, i.e.,
they increase the maximal gap $\overline\GG^{\infty}_n$: this is due to a bad combination between the shape of the forcing term $g$ and the location of $D$.
For example, we observe that the reinforcement $D^0$ improves the performance when $g(x)=\sin(nx)$ with $n=1,2$, while for other values of $n$ the
torsional performance is worse than that of the unstiffened plate (with $D=\emptyset$).
We also observe that there are some ``anomalous values'' of
$\overline\GG^{\infty}_n$, see e.g., the values corresponding to $D^1$ and $n=7$ or $n=9$: they are
considerably larger than the other values in the same column and the reason is again that the place where $D$ acts interacts badly with $g$. In particular,
we notice that both $\sin(7x)$ and $\sin(9x)$ have the same sign in correspondence of $x=\frac{\pi}{4},\frac{\pi}{2},\frac{3\pi}{4}$ that are the centers of
the three vertical arms of $D^1$; in particular,
$g(x)=\sin(7x)\ \Longrightarrow\ g(\tfrac{\pi}{4})=g(\tfrac{3\pi}{4})=-\tfrac{1}{\sqrt2 }\, ,\ g(\tfrac{\pi}{2})=-1$,
see the right picture in Figure \ref{seni}.

Next, we exploit Theorem \ref{gap_limit_swiss} to solve analytically the maxmax problem \eqref{GGD} when $D$ reduces to one horizontal bar (including the free plate). In general, maximizing a Fourier series is a tricky problem that can be solved only for particular choices
of the coefficients, see e.g., \cite{polya}. This is why we focus on the set $\Gamma$ of functions $g=g(x)$ satisfying one of the following:
\begin{itemize}
\item[$\bullet$] $g(x)=\sin(m x)$ with $m\in \N$.
\item[$\bullet$] $g(x)=\displaystyle {\sum_{m\geq N} \secondo_m \sin(mx)}$ with $\{\secondo_m\}_{m\in \N} \subset \ell^2$ and $N\in \N$ large enough.
\item[$\bullet$] $g(x)=\sin(mx)+\sin(3m x)$ with $m \in \N$.
\item[$\bullet$]  $g(x)=\displaystyle {\sum_{m=1}^N  \sin((2m-1)x)}$ with $N\in \N$ sufficiently large.
\end{itemize}
Then, we define the class
$${\mathcal F}_\Gamma:=\left\{f:\, f=\lim_{\alpha\to\infty}f_\alpha \mbox{ in }\hh(\Omega)\,, \mbox{ with $f_\alpha$ as in \eq{f2sinh} and $g\in\Gamma$}\, \right\}$$
and, in Section \ref{expl sol}, we prove the following.

\begin{theorem}\label{examples}
Let ${\mathcal F}={\mathcal F}_\Gamma$ be as above and assume that $D=]0,\pi[\times ]-\varepsilon,\varepsilon[$ for some $0<\varepsilon<\ell$ (i.e., $\mu=0$ in \eqref{Deps}). Then, the solution of the maxmax problem \eqref{GGD} is given by
$\left[\lim_{\alpha\to\infty}R_\alpha \sinh(\alpha y)\right]\sin(x)$, where the limit is in $\hh(\Omega)$.
\end{theorem}

Theorem \ref{examples} states that the worst case as $\alpha\to\infty$ corresponds to the function $g(x)=\sin (x)$. However, $g(x)=\sin (x)$ seems not to be the worst case in
general: to see this, compare the values of $\overline\GG^{\infty}_1$ given in Table \ref{tabella1} with $\GG^\infty_{T_{\pi/2}}$ given in Table~\ref{tabella1bis}.

\section{Weakening Resonant Forces with Polygonal Reinforcements}\label{moregen}

In this section, we intend to study numerically the gap function \eqref{funzionale} and the related minimaxmax problem \eqref{minimaxmax2} in the case the class $\F$ contains some ``resonant-type" force $f$ and the class $\D$ contains ``not-so-nice'' domains $D\in \mathcal L$ (see Definition \ref{spazi}).  Hence,
we minimize the energy \eqref{energy-f}.

Throughout this section, we fix $\ell=\pi/150$ and $\sigma=0.2$ (two reasonable values for plates modeling the deck of a bridge, see \cite{befega}).
Moreover, we assume that $\tanh(\sqrt{2}m\ell)>\sigma^2/(2-\sigma)^2 \sqrt{2}m\ell$ so that $m\le2734$. Then, for any integer $m\le2734$, the first torsional eigenfunction $\overline e_m$ of $\Delta^2$ with the boundary conditions in \eqref{eq:Poisson_CROSS} having $m-1$ nodes in the $x$-direction and the corresponding eigenvalue $\nu_{m}$ are known; see \cite{fergaz}.
Notice that
$m^4<\nu_{m}<\big(m^2+\tfrac{\pi^2}{4\, \ell^2}\big)^2$.
A detailed analysis of the variation of all the eigenvalues under domain deformations was performed in \cite{bebuga}.
We aim to study the effect of a reinforcement $D$ when the force $f$ is at resonance, namely proportional to a torsional (odd) eigenfunction: we take
$f=\overline{e}_{m}(x,y)$. 
For these functions $f$ we then deal with problem \eqref{loadpb} and we seek the best shape of the reinforcement $D$ in order to lower the maximal gap $\GG_D^{\infty}$. We numerically study problem \eq{minimaxmax2} within classes of forces  (with $m$ from $1$ to $5$) and of reinforcements $\D$ of sets composed by two parallel strips, by triangles, by squares, and by hexagons as in Figure \ref{tri}:
\begin{equation}\label{FDn2}
{\mathcal F}=\{\overline{e}_{1},...,\overline{e}_{5}\}\quad \text{and} \quad  {\mathcal D}=\{\text{Strips, Triangles, Squares, Hexagons}\}\,.
\end{equation}

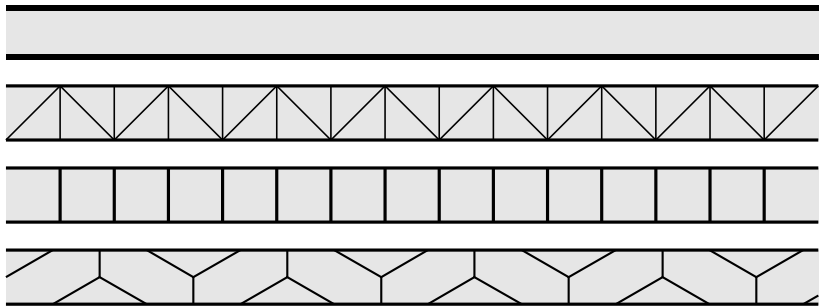
\begin{figure}[t]
\begin{center}

\begin{tikzpicture}[scale=0.36]
\fill[color=black!10!] (0,-1) rectangle (30,1);
\fill[color=black] (0,-1) rectangle (30,-0.8) ;
\fill[color=black] (0,0.8) rectangle (30,1);
\end{tikzpicture}

\par\bigskip\par

\begin{tikzpicture}[scale=0.36]

\fill[color=black!10!] (0,-1) rectangle (30,1);
\draw[very thick] (0,-1)-- (30,-1);
\draw[very thick] (0,1)-- (30,1);

\draw[semithick] (2,-1)-- (2,1);
\draw[semithick] (4,-1)-- (4,1);
\draw[semithick] (6,-1)-- (6,1);
\draw[semithick] (8,-1)-- (8,1);
\draw[semithick] (10,-1)-- (10,1);
\draw[semithick] (12,-1)-- (12,1);
\draw[semithick] (14,-1)-- (14,1);
\draw[semithick] (16,-1)-- (16,1);
\draw[semithick] (18,-1)-- (18,1);
\draw[semithick] (20,-1)-- (20,1);
\draw[semithick] (22,-1)-- (22,1);
\draw[semithick] (24,-1)-- (24,1);
\draw[semithick] (26,-1)-- (26,1);
\draw[semithick] (28,-1)-- (28,1);

\draw[semithick] (0,-1)-- (2,1);
\draw[semithick] (2,1)-- (4,-1);
\draw[semithick] (4,-1)-- (6,1);
\draw[semithick] (6,1)-- (8,-1);
\draw[semithick] (8,-1)-- (10,1);
\draw[semithick] (10,1)-- (12,-1);
\draw[semithick] (12,-1)-- (14,1);
\draw[semithick] (14,1)-- (16,-1);
\draw[semithick] (16,-1)-- (18,1);
\draw[semithick] (18,1)-- (20,-1);
\draw[semithick] (20,-1)-- (22,1);
\draw[semithick] (22,1)-- (24,-1);
\draw[semithick] (24,-1)-- (26,1);
\draw[semithick] (26,1)-- (28,-1);
\draw[semithick] (28,-1)-- (30,1);
\end{tikzpicture}

\par\bigskip\par

\begin{tikzpicture}[scale=0.36]

\fill[color=black!10!] (0,-1) rectangle (30,1);
\draw[very thick] (0,-1)-- (30,-1);
\draw[very thick] (0,1)-- (30,1);

\draw[very thick] (2,-1)-- (2,1);
\draw[very thick] (4,-1)-- (4,1);
\draw[very thick] (6,-1)-- (6,1);
\draw[very thick] (8,-1)-- (8,1);
\draw[very thick] (10,-1)-- (10,1);
\draw[very thick] (12,-1)-- (12,1);
\draw[very thick] (14,-1)-- (14,1);
\draw[very thick] (16,-1)-- (16,1);
\draw[very thick] (18,-1)-- (18,1);
\draw[very thick] (20,-1)-- (20,1);
\draw[very thick] (22,-1)-- (22,1);
\draw[very thick] (24,-1)-- (24,1);
\draw[very thick] (26,-1)-- (26,1);
\draw[very thick] (28,-1)-- (28,1);
\end{tikzpicture}

\par\bigskip\par

\begin{tikzpicture}[scale=0.36]

\fill[color=black!10!] (0,-1) rectangle (30,1);
\draw[very thick] (0,-1)-- (30,-1);
\draw[very thick] (0,1)-- (30,1);

\draw[thick] (0,0) -- (1.73,1);
\draw[thick] (1.73,-1) -- (3.46,0) -- (3.46,1);
\draw[thick] (3.46,0) -- (5.20, -1);
\draw[thick] (5.20,1) -- (6.92,0) -- (6.92,-1);
\draw[thick] (6.92,0) -- (8.66, 1);
\draw[thick] (8.66,-1) -- (10.39,0) -- (10.39,1);
\draw[thick] (10.39,0) -- (12.12, -1);
\draw[thick] (12.12,1) -- (13.86,0) -- (13.86,-1);
\draw[thick] (13.86,0) -- (15.58, 1);
\draw[thick] (15.58,-1) -- (17.3,0) -- (17.3,1);
\draw[thick] (17.3,0) -- (19.05, -1);
\draw[thick] (19.05,1) -- (20.78,0) -- (20.78,-1);
\draw[thick] (20.78,0) -- (22.52, 1);
\draw[thick] (22.52,-1) -- (24.25, 0) -- (24.25,1);
\draw[thick] (24.25,0) -- (25.98, -1);
\draw[thick] (25.98,1) -- (27.71,0) -- (27.71,-1);
\draw[thick] (27.71,0) -- (29.44, 1);
\draw[thick] (29.44,-1) -- (30,-0.68);
\end{tikzpicture}
\caption{Qualitative patterns of the trusses in $\D$.}\label{tri}
\end{center}
\end{figure}

The black lines are the thick stiffening trusses $D$ put below the plate and their
total area is constant. More precisely, the first plate is reinforced by two parallel trusses of width $X=\frac{1046\pi}{750^2}\approx0.00584$,
thereby having a global area of $2\pi X$. The three remaining shapes all have two parallel trusses of width $W=\frac \pi {750}$ along the free edges of
the plate for a total area of $2\pi W$, while the remaining area of
$2\pi(X-W)$ is distributed in connecting transverse trusses which generate some polygons all along the plate, see again Figure \ref{tri}. The triangular transverse truss is composed of $74$ vertical segments having length $2\ell-2W$ and $75$ oblique segments having length $(\pi/75 -2W)\sqrt{2}$, both having width $0.00287159$ (approximately). The squared transverse truss is composed of $74$ vertical segments having length $2\ell-2W$ and width $W$. Finally, the hexagonal transverse truss is composed of $17$ $Y$-shaped components, alternating upwards and downwards, complemented by two segments at the opposite ends of the plate (playing the role of the oblique branches of $Y$), whose measures are $\ell-W$ for the length of the vertical legs and $Z=0.0215211$ (approximated) for the width. These four designs, all belonging to the class $\mathcal L$ of Definition \ref{spazi}, have their own motivation.
The first one is the most natural, putting reinforcements only on the two free edges. The triangular truss is the most frequently used by engineers. The third one is also natural, putting the simplest transverse connections between the free edges. Finally, a truss composed of regular hexagons was shown to have better bending performances in \cite{hexagon} where the \lq \lq boundary effects" were neglected. In fact, what really counts is to have angles of size $2\pi/3$, as in \emph{irrigation} or \emph{traffic problems}, see \cite{but1,but2,but3}.
Let us also mention that it has been known since the 19th century that soap bubbles reach an equilibrium on flat surfaces when
the angles between three adjacent bubbles are always $2\pi/3$, see \cite{Th87}.
This angle has the peculiarity to \lq \lq optimize the distances'' and it is therefore interesting to measure its performance also in stiffening trusses.
The numerical values for the maximal gap are reported in Table \ref{table2}.
\begin{table}[h]
\begin{center}
\caption{Numerical values of $10^4\times \GG^{\infty}_{\overline e_{m},D}$ for the different polygonal reinforcements $D$ and resonant forces $\overline{e}_{m}$ (with $\ell=\pi/150$, $\sigma=0.2$, and $d=2$).}\label{table2}
\vspace{.1 in}
\begin{tabular}{|c|c|c|c|c|c|c|c|c|c|c|c|}
\hline
 \!&\! $\overline{e}_{1}$\!&\! $\overline{e}_{2}$\!&\! $\overline{e}_{3}$\!&\! $\overline{e}_{4}$ \!&\! $\overline{e}_{5}$ \\
\hline
$ \emptyset$\!&\! $43.629$ \!&\! 21.811\!&\! 14.537\!&\! 10.899 \!&\! 8.7147 \\
\hline
 Strips \!&\!$25.448$\!&\!$6.3602$\!&\!$2.8255$\!&\!$1.5883$\!&\! $1.0157$\\
\hline
Triangles \!&\!$29.363$\!&\!$7.2105$\!&\!$3.2643$\!&\!$1.8409$\!&\! $1.1855$\\
\hline
Squares \!&\!$27.946$\!&\!$6.9846$\!&\!$3.1028$\!&\!$1.7442$\!&\! $1.1154$\\
\hline
Hexagons \!&\!$28.875$\!&\!$7.1787$\!&\!$3.2007$\!&\!$1.7919$\!&\! $1.1304$\\
\hline
\end{tabular}
\end{center}
\end{table}

The class $\D$ that we have introduced here could be enlarged by considering also other geometries for $D$. Regarding the hexagonal design, we actually studied different positions of the intersections in the Y-shaped elements. The results contained in Table \ref{table2} are given for elements where the intersections occur on the midline of the plate, hence with the vertical branch having length $\ell$, while we performed computations also for cases where the vertical branch is longer or shorter than $\ell$. Even though one might expect the gap functions to be monotone or to have a unique minimum point (with respect to the length of the vertical branch), this does not occur, the behavior being very specific depending on the particular resonant force $\overline e_{m}$ considered. In some cases, the maximal gap exits the range we saw in Table \ref{table2}: for $\overline e_{1}$ and $\overline e_{3}$ the gap function is always bounded by that of the squares and that of the triangles, while for $\overline e_{2}$ the branch $b$ of length $b=2\ell/3$ produces a situation worse than the triangles, for $\overline e_{4}$ the cases $b=4\ell/3, 11\ell/8$ perform better than the squares. For $\overline e_{5}$ the case $b=23\ell/20$ performs worse than the triangles while the cases $b=4\ell/3, 17\ell/20$ are better than the squares.

\section{Proofs of the Existence Results}\label{sec8}


We first prove the continuity of the map defined in \eqref{funzionale}. We recall that in all the cases considered for the class $\F$, the weak* topology coincides with the weak topology, except when $\F=L^\infty(\Omega)$.

\begin{proposition}\label{trivial}
Let $\F$ be either $\hh(\Omega)$ (for $\mathbb{E}_1$) or $L^p(\Omega)$ with $p\in ]1,+\infty]$ (for both $\mathbb{E}_1$ and $\mathbb{E}_2$). Let also $\D$ be a class of open subdomains of $\Omega$ closed with respect to the $L^1$ topology. Then the map
$\GG^\infty_{f,D}\, :\, \F\times\D \to[0,\infty[$ with $(f,D)\mapsto \GG^\infty_{f,D}$
is sequentially continuous when $\F$ is endowed with the weak* topology and $\D$ is endowed with the $L^1$ topology.
\end{proposition}
\begin{proof} Let $\{(f_n,D_n)\}_n \subset \F\times\D$ be such that $(f_n,D_n)\rightarrow (f,D)$ as $n\rightarrow+\infty$, hence $f_n\rightharpoonup^* f$ in $\F$ and $\chi_{D_n} \rightarrow \chi_{D} $ in $L^1$ as $n\rightarrow +\infty$. We denote by $u=u_{f,D}$ and $u_n=u_{f_n,D_n}$ the corresponding solutions of \eqref{weakD}. Recalling \eqref{scalarpD}, \eqref{weakD} with $f=f_n$ and $D=D_n$ reads
\neweq{subtraction}
(u_n,v)_{H^2_*}+d(u_n,v)_{D_n}= \langle f_n,v\rangle \qquad\forall v\in H^2_*(\Omega)\, .
\endeq
Since $f_n\rightharpoonup^* f$ in $\F$, the above equality with $v=u_n$ yields $\|u_n\|_{H^2_*}\leq C $ for some $C>0$. In particular, $u_n\rightharpoonup \bar u$ up to a subsequence in $H^2_*(\Omega)$ for some $\bar u\in H^2_*(\Omega)$.
Next, by adding and subtracting $d(u_n,v)_{D}$ in \eqref{subtraction}, we obtain that, for every $v\in H^2_*(\Omega)$
\neweq{subtraction2}
(u_n,v)_{H^2_*}+d(u_n,v)_{D}+d(u_n,v)_{D_n\setminus D}-d(u_n,v)_{D\setminus D_n}= \langle f_n,v\rangle \, .
\endeq
Since $\chi_{D_n} \rightarrow \chi_{D} $ in $L^1(\Omega)$ yields $|D_n\triangle D| \rightarrow 0$ as $n\rightarrow +\infty$, we deduce that
$|(u_n,v)_{D_n\setminus D}|\leq C \|v\|_{H^2(D_n\setminus D)}=o(1)$ { as } $n\rightarrow +\infty$
and similarly $(u_n,v)_{D \setminus D_n}=o(1)$.  By this, passing to the limit in \eqref{subtraction2}, we conclude that
$(\bar u,v)_{H^2_*}+d(\bar u,v)_{D}= \langle f,v\rangle$ for all $v\in H^2_*(\Omega)$;
hence $\bar u\equiv u$. Furthermore, from the compactness of the embedding $H^2_*(\Omega)\subset C^0(\overline{\Omega})$, we obtain $u_n\rightarrow u $ in $C^0(\overline{\Omega})$.
In terms of the gap functions, this means that $\GG_{f_n,D_n}(x)$ converges uniformly to $\GG_{f,D}(x)$ as $n\rightarrow +\infty$ over $[0,\pi]$. In particular, $\GG^\infty_{f_n,D_n} \rightarrow \GG^\infty_{f,D}$ as $n\rightarrow +\infty$. This concludes the proof.\end{proof}

\begin{proof}[Theorem \ref{optimal}]
Fix $D \subset \Omega$. If $\{f_n\}\subset \hh(\Omega)$ is a maximizing sequence for \eqref{functional}, since $\|f_n\|_{\hh}=1$, up to a subsequence, we have $f_n\rightharpoonup \overline f$  in $\hh(\Omega)$. By  Proposition \ref{trivial}, $\max \big \{  \GG^\infty_{f,D}: \text{$f\in \hh(\Omega)$, $\|f\|_{\hh}=1$}\big\}=\GG^\infty_{\overline f,D}$.
Moreover, it must be $\|\overline f\|_{\hh}=1$. Otherwise, if $\|\overline f\|_{\hh}<1$, set $\widetilde f={\overline f}/{\|\overline f\|_{\hh}}$ and by linearity we get $\GG^\infty_{\widetilde f, D}={\GG^\infty_{\overline f,D}}/{\|\overline f\|_{\hh}}>\GG^\infty_{\overline f,D}$,
a contradiction that proves the first part of Theorem \ref{optimal}.

Now, let $\{f_n\} \subset L^p(\Omega)$ be a maximizing sequence for \eqref{maxim} such that $\|f_n\|_{L^p}=1$. Up to a subsequence and for some $\overline f$, we have $f_n\rightharpoonup \overline f$ in $L^p(\Omega)$ if $1<p<\infty$ and $f_n\rightharpoonup ^* \overline f$ in $L^{\infty}(\Omega)$. In particular,  by lower semicontinuity of the norms with respect to these convergences, $\|\overline f\|_{L^p}\leq \|f_n\|_{L^p}=1$. Moreover, by Proposition \ref{trivial}, we have $\max \big \{  \GG^\infty_{f,D}:\, \text{$f\in L^p(\Omega)$ with $\|f\|_{L^p}=1$}\big\} =\GG^\infty_{\overline f,D}$.
Finally, the proof that $\|\overline f\|_{L^p}=1$ follows by arguing as above.
\end{proof}

\begin{proof}[Theorem \ref{optimaltruss}]
Using the Direct Method of the Calculus of Variations, it is sufficient to find a topology for which the functional
$D\mapsto\GG_D^\infty$ defined in \eqref{functional} is lower semicontinuous while the class of admissible sets $\D$ is compact. For this purpose we use the $L^1$-convergence of sets, namely the $L^1$-convergence of the characteristic functions associated to the sets. Indeed, by its definition \eqref{functional}
and the continuity proved in Proposition \ref{trivial}, it follows that the functional $\GG^\infty$ is lower-semicontinuous with respect to the
$L^1$-convergence of sets. Therefore, it remains to prove that the classes introduced in Definition \ref{spazi} are compact with respect to this
convergence: we do it for each class.

(a) Consider a sequence of crosses $\{D_n\}$ in $\mathcal C$:
by the Bolzano-Weierstrass Theorem the sequences of points $\{x_n^i\}$ and $\{y_n^i\}$ converge,  up to subsequences, to some $x^i\in  [\mu, \pi-\mu]$, $i=1,\dots, N$, and some $y^j\in[-\ell+\eps,\ell-\eps]$, $j=1,\dots,M$, respectively. By the Lebesgue Dominated Convergence Theorem, it turns out that $|D_n\triangle D|\to0$ as $n\to \infty$ where $D$ is the cross
\[
\Big(\bigcup_{i=1}^N (x^i-\mu,x^i+\mu\times ]-\ell,\ell[\Big)\cup\Big(\bigcup_{j=1}^M ]0,\pi[\times ]y^j-\varepsilon,y^j+\varepsilon[ \Big)\,;
\] this means that $\chi_{D_n}\to\chi_{D}$ in $L^1(\Omega)$ as $n\to \infty$. Moreover, $|D|=\kappa$, thanks to the area constraint. Therefore, the class $\mathcal C$ with area constraint is compact with
respect to the $L^1$-convergence of sets.

(b) To each rectangle $R\subset\Omega$ we associate its four vertices $V_1(R),\dots,V_4(R)$ in such a way that $V_1(R)$
is the upper-right vertex (i.e., the one with largest $y$-coordinate in the case such a vertex is unique, otherwise the one with largest $x$-coordinate) and the remaining $V_i(R)$ are ordered clockwise.
Consider a sequence of rectangles $\{R_n\}$ all having inradius at least $\varepsilon$:
by the Bolzano-Weierstrass Theorem, up to extracting a subsequence (that we do not relabel), the sequence of vertices $\{V_1(R_n)\}$ converges to some point $V_1\in\overline{\Omega}$. Up to extracting a further subsequence, the sequence of vertices $\{V_2(R_n)\}$ also converges to some point $V_2\in\overline{\Omega}$. Repeating this argument for the remaining vertices, we infer that each of the four sequences of vertices $\{V_i(R_n)\}$  converges, up to subsequences, to some point $V_i\in\overline{\Omega}$ (for $i=1,2,3,4$). Let $R$ be the open convex hull of the four points $V_1,\dots,V_4$; since, by construction, the scalar product of two consecutive sides is $(V_i(R_n)V_{i+1}(R_n),V_{i+1}(R_n)V_{i+2}(R_n))=0$ for $i=1,2,3,4$, where we set $V_5(R):=V_1(R)$ and $V_6(R):=V_2(R)$, passing to the limit as $n\to\infty$, and using the continuity of the scalar product it follows that $(V_iV_{i+1},V_{i+1}V_{i+2})=0$ (for $i=1,2,3,4$, where $V_5:=V_1$ and $V_6:=V_2$). Moreover, since the distance between two consecutive vertices of $R_n$ is larger than $2\varepsilon$ for all $n$, also the inradius of R is at least $\varepsilon$.  This implies that the set $R$ is an open rectangle having the distinct vertices $V_i(R)=V_i$ (for $i=1,2,3,4$). Moreover,  by the Lebesgue Dominated Convergence Theorem, it turns out that
$|R_n\triangle R|\to0$ as $n\to \infty$; this means that $\chi_{R_n}\to\chi_R$ in $L^1(\Omega)$ as $n\to \infty$.\par
Then take a sequence of sets $D_n\in\mathcal T$ with $D_n=\cup_{i=1}^N R_n^i$. Using the argument above, up to subsequences, we have that
$|R_n^i\triangle R^i|\to0$ as $n\to \infty$ for some rectangles $R^i$ all having inradius at least $\varepsilon$.
Hence, $\chi_{R_n^i}\to\chi_{R^i}$ in
$L^1(\Omega)$ for all $i=1,...,N$ and, in turn, $\chi_{D_n}\to\chi_{D}$ in $L^1(\Omega)$. The area constraint yields that $|D|=\kappa$.
Therefore, the class $\mathcal T$ with area constraint is compact with respect to the $L^1$-convergence of sets.

(c) Let $\Sigma_n$ be a sequence of closed connected sets with $\mathcal H^1(\Sigma_n)\leq L$. From the Blaschke Selection Theorem and the Go\l ab Theorem (see e.g. \cite[Theorem 4.4.17]{ambtil}), up to a subsequence we know that $\Sigma_n\to \Sigma$ with respect to the Hausdorff distance, where $\Sigma$ is a closed and connected set with $\mathcal H^1(\Sigma)\leq L$. Then the distance function to $\Sigma_n$ converges to the distance function to $\Sigma$ uniformly on $\overline \Omega$. This, with the fact that the Lebesgue measure of the set $\partial K^\e=\{x\in \Omega : \, \text{dist}_K=\e\}$ is zero, implies that $K_n^\e$ converges in $L^1$ to $K^\e$ (see \cite{hepi}).

(d) Using again \cite[Theorem 2.4.10]{hepi} we obtain the compactness with respect to the $L^1$ convergence of the space $\mathcal L$ with area constraint.
\end{proof}

\section{Proofs of the Symmetry Results}

\begin{proof}[Theorem \ref{fsymmetry}]\label{pp}
Let $f\in\hh(\Omega)$ be such that $\|f\|_{\hh}=1$ and consider the solution $u_f\in H^2_*(\Omega)$ of \eqref{weakD}. Since $D$ is symmetric,
following the decomposition \eqref{scomp} we may rewrite \eqref{weakD} as
\neweq{doppia}
(u_f^e,v^e)_{H^2_*}+(u_f^o,v^o)_{H^2_*}+d(u_f^e,v^e)_{D}+d(u_f^o,v^o)_{D}=\langle f^e,v^e\rangle+\langle f^o,v^o\rangle\, ,
\endeq
for all $v\in H^2_*(\Omega)$. Moreover, by \eqref{funzionale}, we have $\GG_{f,D}(x)=u^o_f(x,\ell)-u^o_f(x,-\ell)$ and also that $\GG^\infty_{f,D}=\max_{x\in[0,\pi]}\ \big|u^o_f(x,\ell)-u^o_f(x,-\ell)\big|$. In particular, if $f^o=0$ then $u^o=0$ and $\GG^\infty_{f,D}=0$ so that $f$ cannot be a maximizer for $\GG^\infty_{f,D}$.
Hence, by \eqref{normf}, there exists $0<\alpha\le1$ such that $\alpha=\|f^o\|_{\hh}\le\|f\|_{\hh}=1$. Consider now the problem
$(w,v)_{H^2_*}+d(w,v)_{D}=\frac{1}{\alpha}\langle f^o,v\rangle$ for all $v\in H^2_*(\Omega)$.
By linearity and by \eqref{doppia}, its solution is $w=u^o/\alpha$, then
$ \GG_{\frac{f^o}{\alpha},D}(x)=\frac{1}{\alpha}\GG_{f,D}(x)$ and  $\GG^\infty_{\frac{f^o}{\alpha},D}=\frac{1}{\alpha}\GG^\infty_{f,D}\ge\GG^\infty_{f,D}\, .$

Hence, we have shown that for all $f\in\hh(\Omega)$ such that $\|f\|_{\hh}=1$, there exists $g\in\hho$ such that $\|g\|_{\hh}=1$ ($g=f^o/\alpha$) and
$\GG^\infty_{g,D}\ge \GG^\infty_{f,D}$. This proves the first part of Theorem \ref{fsymmetry}.\par\smallskip

The remaining part of Theorem \ref{fsymmetry} follows the inverse path. Let $f$ be as in the statement and take any
$\phi\in\hhe$ such that $\|\phi\|_{\hh}\le1$. Then, put $g=f+\phi$ so that $g^o=f$ and $g^e=\phi$. By \eqref{normf} we have $\|g\|_{\hh}=1$.
By slightly modifying the arguments above we see that $\GG^\infty_{g,D}=\GG^\infty_{f,D}$.
\end{proof}

For the proof of Theorem \ref{fsymmetry2} we need the following result.

\begin{lemma}\label{inequality}
Let $1\le p\le\infty$ and $a>0$. If  $\phi\in L^p(]-a,a[)$ then
\neweq{large}
\|\phi^o\|_{L^p}\le\|\phi\|_{L^p}\, .
\endeq
Moreover:\par
 -- if $p=1$ then the inequality in \eqref{large} is strict if and only if $|\phi^o(x)|<|\phi^e(x)|$ in a subset of $]-a,a[$ of positive measure;\par
 -- if $1< p <\infty$ then the inequality in \eqref{large} is strict if and only if $\phi$ is not odd ($\phi\not\equiv\phi^o$);\par
 -- if $p=\infty$ then the inequality in \eqref{large} is strict if and only if for any $\{x_n\}\subset ]-a,a[$ such that $|\phi(x_n)|\to\|\phi\|_{L^\infty}$,
one has $\liminf_n|\phi^e(x_n)|>0$; in particular, if $\phi\in C^0[-a,a]$, then the inequality is strict if and only if $\phi^e(\overline{x})\neq0$
in every point $\overline{x}$ where $|\phi|$ attains its maximum.
\end{lemma}
\begin{proof} Since $\phi^o(x)=\tfrac{\phi(x)-\phi(-x)}{2}$,
the inequality \eqref{large} follows from the Minkowski inequality and the symmetry of $]-a,a[$.\par
If $p=1$, then the Minkowski inequality, just used to obtain \eqref{large}, reads
$$\int_{-a}^a|\phi^o(x)|\, dx=\frac12\int_{-a}^a|\phi(x)-\phi(-x)|\, dx\le\int_{-a}^a\tfrac{|\phi(x)|+|\phi(-x)|}2 \, dx\!=\!\int_{-a}^a|\phi(x)|\, dx$$
so that it reduces to an equality if and only if
$$0 \ge \phi(x)\phi(-x)=\Big[\phi^e(x)+\phi^o(x)\Big]\Big[\phi^e(-x)+\phi^o(-x)\Big]=\phi^e(x)^2-\phi^o(x)^2$$
for a.e. $x\in ]-a,a[$.
This means that $|\phi^e(x)|\le|\phi^o(x)|$ for a.e.\ $x\in ]-a,a[$. Since this is a necessary and sufficient condition, the statement for $p=1$ is proved.\par
If $p\in ]1,+\infty[$, the Minkowski inequality is itself obtained via an application of H\"older's inequality and equality holds if and only if the
two involved functions are multiples of each other. In the present situation, this means that $\phi(x)=\alpha\phi(-x)$ for some $\alpha=\alpha(p)<0$ and
for a.e.\ $x\in ]-a,a[$. The only possibility is that $\alpha=-1$, which means that $\phi=\phi^o$ and $\phi^e\equiv0$.
Since this is a necessary and sufficient condition, also the statement for $p>1$ is proved.\par
If $p=\infty$, we claim that equality holds in \eq{large} if and only if there exists $\{x_n\}\subset]-a,a[$ such that $|\phi(x_n)|\to\|\phi\|_{L^\infty}$
and $\phi^e(x_n)\to0$. Indeed, if such a sequence exists, then $|\phi^o(x_n)|=|\phi(x_n)-\phi^e(x_n)|\to\|\phi\|_{L^\infty}$ which proves that
$\|\phi^o\|_{L^\infty}=\|\phi\|_{L^\infty}$. Conversely, if equality holds then there exists $\{x_n\}\subset]-a,a[$ such that
$\phi^o(x_n)\to\|\phi\|_{L^\infty}$. This yields
$\phi(x_n)-\phi^e(x_n)=\phi^o(x_n)\to\|\phi\|_{L^\infty}$ and $\phi^e(x_n)-\phi(-x_n)=-\phi^o(-x_n)=\phi^o(x_n)\to\|\phi\|_{L^\infty}$,
which proves that $\phi^e(x_n)\to0$ since otherwise one of $|\phi(\pm x_n)|$ would tend to exceed $\|\phi\|_{L^\infty}$. The claim is so proved and therefore
the strict inequality occurs in the opposite situation: this proves the first statement.\par
In the case, where $\phi\in C^0([-a,a])$, the sequences just used to prove the statement may be replaced by their limits.
\end{proof}

\begin{proof}[Theorem \ref{fsymmetry2}] For every $p\in]1,\infty]$, \eqref{large} combined with the argument in the proof of Theorem \ref{fsymmetry} yields that a maximizer $f$ can be sought as an odd function.\par
If $1<p<\infty$, by contradiction, let $f\in L^{p}(\Omega)$ such that $\|f\|_{L^{p}}=1$ be a non-odd maximizer. Since $f\not\equiv f^o$, by Lemma \ref{inequality} $\|f^o\|_{L^p}<\|f\|_{L^p}=1$. Take now $\overline f={f^o}/{\|f^o\|_{L^p}}$,  recalling that $f^e$ plays no role in the value of the gap function, we obtain $\|\overline f\|_{L^{p}}=1$ and $\GG^\infty_{\overline f,D}=\GG^\infty_{f,D}/{\|f^o\|_{L^p}}>\GG^\infty_{f,D}$, a contradiction.\par
If $p=\infty$, take an odd function $f\in L^{\infty}(\Omega)$ such that $\|f\|_{L^{\infty}}=1$ and
$\GG^\infty_{f,D}=\max\Big\{\GG^\infty_{\phi,D}:\, \text{$\phi\in L^{\infty}(\Omega)$ with $\|\phi\|_{L^{\infty}}=1$}\Big\}$.
If $|f(x,y)|<1$ on a subset $\omega\subset\Omega$ of positive measure, take any even function $h$ such that $h(x,y)\equiv0$ in $\Omega\setminus\omega$
and $|h(x,y)|<1-|f(x,y)|$ in $\omega$. Then, $g=f+h$ is not odd and satisfies $\|g\|_{L^\infty}=1$, $g^e=h\neq0$, and
$\GG^\infty_{g,D}=\GG^\infty_{f,D}$ (by linearity since $\GG^\infty_{h,D}=0$).
\end{proof}

\section{Proofs of Theorems \ref{gap_limit_swiss} and \ref{examples}}\label{expl sol}

We prove Theorem \ref{gap_limit_swiss} and Theorem \ref{examples} in several steps.
Let $g\in L^2(]0,\pi[)$, $\alpha\ge0$ with $\alpha\not\in\mathbb{N}$ and
\neweq{f2}
k_\alpha(x,y)=K_\alpha e^{\alpha y}g(x),\quad
K_\alpha:=\frac{\alpha}{2C_g\sinh(\alpha \ell)}\text{ and } C_g:=\int_0^{\pi}|g(x)|\,dx\,,
\end{equation}
so that $\|k_\alpha\|_{L^1}=1$. Let $h_\alpha:=\frac{k_\alpha}{1+d\chi_{D^N}}$; we first focus on the \emph{auxiliary problem}
\begin{equation}\label{eq:Poisson_CROSS}
\begin{cases}
\Delta^2 w=h_\alpha\,, & \qquad \text{in } \Omega\,, \\
w=w_{xx}=0\,, & \qquad \text{on } \{0,\pi\}\times ]-\ell,\ell[\,, \\
w_{yy}+\sigma w_{xx}=w_{yyy}+(2-\sigma)w_{xxy}=0\,, & \qquad \text{on } ]0,\pi[\times \{-\ell,\ell\}\,.\\
\end{cases}
\end{equation}
Indeed, if $w_\alpha$ solves \eq{eq:Poisson_CROSS}, then a multiple of its odd part
$$u_\alpha(x,y):=\frac{R_\alpha}{K_\alpha}w^o_\alpha(x,y)=\frac{\sinh(\alpha \ell)}{2(\cosh(\alpha \ell)-1)}\,(w_\alpha(x,y)-w_\alpha(x,-y))$$
solves problem \eqref{loadpb} with $f=f_\alpha$ and $D=D^N$. Moreover, if $\GG_\alpha(x)$ is the gap function corresponding to $w_\alpha$, then $\tfrac{R_\alpha}{K_\alpha}\GG_\alpha(x)$ is the gap function corresponding to $u_\alpha$.
Therefore, since $\frac{R_\alpha}{K_\alpha} =1+2e^{-\alpha\ell}+o(e^{-\alpha\ell})$ as $\alpha \rightarrow + \infty$,
the limit of the gap function corresponding to $u_\alpha$ and the asymptotic behavior of the corresponding coefficients are exactly the same as those for $w_\alpha$.

Now, we focus on the explicit solution of the auxiliary problem.
We expand $g\in L^2(]0,\pi[)$ in a Fourier series
\neweq{serieg}
g(x)=\serie\secondo_m\smx\, ,\qquad\secondo_m=\frac{2}{\pi}\ipi g(x)\smx\, dx\,.
\endeq
Then, if we set
\neweq{INN}
I^N:=\bigcup_{ i=1}^{2N+1}
\left(\frac{\pi i}{2N+2}-\frac{\mu}{2N+1},\frac{\pi i}{2N+2}+\frac{\mu}{2N+1}\right)\,,
\endeq
for every $x\in ]0,\pi[$ and every $y\in ]-\ell,-\varepsilon[\cup ]\varepsilon,\ell[$, we have
\neweq{serietilde}
\frac{g(x)}{1+d\chi_{D^N}(x,y)}=\frac{g(x)}{1+d\chi_{I^N}(x)}=\serie\primo_m\smx\,,
\endeq
where the coefficients $\primo_m$ are as defined in \eqref{gtilde},
while if we set $\widehat\gamma_m=\frac{\secondo_m}{1+d}$, we have
$\frac{g(x)}{1+d\chi_{D^N}(x,y)}=\frac{g(x)}{ 1+d}=\serie\widehat{\gamma}_m\smx$ for all $x\in ]0,\pi[$ and $y\in ]-\varepsilon,\varepsilon[$.

In the sequel, we will need the following constants (only depending on $m$, $\alpha$ and $\varepsilon$):
\small{
\begin{eqnarray}\label{Fi}
\notag & F_1(\eps)  :=  \tfrac{\alpha(\alpha^2-3m^2)\sinh(m\varepsilon) +2m^3\cosh(m\varepsilon) }{2m^3}+(\alpha^2-m^2)\tfrac{m\sinh(m\varepsilon)-\alpha\cosh(m\varepsilon)}{2m^2}\varepsilon\,,\\\notag
& F_2(\eps)  :=  -\tfrac{\alpha(\alpha^2-3m^2)\cosh(m\varepsilon)+2m^3\sinh(m\varepsilon)}{2m^3} -(\alpha^2-m^2)\tfrac{m\cosh(m\varepsilon)-\alpha\sinh(m\varepsilon)}{2m^2}\varepsilon \,,\\\notag
& F_3(\eps)  :=  (\alpha^2-m^2)\tfrac{\alpha\cosh(m\varepsilon)-m\sinh(m\varepsilon)}{2m^2}  \,,\\
& F_4(\eps)  :=  -(\alpha^2-m^2)\tfrac{\alpha\sinh(m\varepsilon)-m\cosh(m\varepsilon)}{2m^2}\,,
\end{eqnarray}
\begin{equation}\label{Fpm}
\begin{split}
F_1^{\pm}(\eps)&:=F_1(\eps)\pm e^{-2\alpha\varepsilon}F_1(-\eps)\,, \qquad F_2^{\pm}(\eps):=F_2(\eps)\pm e^{-2\alpha\varepsilon}F_2(-\eps)\,,\\
F_3^{\pm}(\eps)&:=F_3(\eps)\pm e^{-2\alpha\varepsilon}F_3(-\eps)\,, \qquad F_4^{\pm}(\eps):=F_4(\eps)\pm e^{-2\alpha\varepsilon}F_4(-\eps)\,,
\end{split}
\end{equation}
\begin{equation}\label{a}
a=a(m,\alpha,\eps):=K_\alpha e^{\alpha\varepsilon}\frac{\primo_m-\widehat\gamma_m}{(m^2-\alpha^2)^2}\,,
\end{equation}
\begin{eqnarray}\label{G}
\notag &G_1:=-\tfrac{a}{2}\big\{(1\!-\!\sigma)m^2\cosh(m\ell)\, F_1^+(\eps)+ m[2\cosh(m\ell)+(1\!-\!\sigma)m\ell\sinh(m\ell)]\, F_4^+(\eps)\\
\notag & + (1\!-\!\sigma)m^2\sinh(m\ell)\, F_2^-(\eps)+  m[2\sinh(m\ell)+(1\!-\!\sigma)m\ell\cosh(m\ell)]\, F_3^-(\eps)\big\}\,\\
\notag  & G_2  :=\tfrac{a}{2}\big\{(1\!-\!\sigma)m^3\cosh(m\ell)\,F_2^-(\eps) - m^2[(1\!+\!\sigma)\cosh(m\ell)-(1\!-\!\sigma)m\ell\sinh(m\ell)]\, F_3^-(\eps)\\
 \notag & +(1\!-\!\sigma)m^3\sinh(m\ell)\, F_1^+(\eps)-  m^2[(1\!+\!\sigma)\sinh(m\ell)-(1\!-\!\sigma)m\ell\cosh(m\ell)]\, F_4^+(\eps)\big\}\, ,\\
\notag  & G_3 :=-\tfrac{a}{2}\big\{(1\!-\!\sigma)m^2\cosh(m\ell)\, F_1^-(\eps)+ m[2\cosh(m\ell)+(1\!-\!\sigma)m\ell\sinh(m\ell)]\, F_4^-(\eps)\\
\notag  &+ (1\!-\!\sigma)m^2\sinh(m\ell)\, F_2^+(\eps)+  m[2\sinh(m\ell)+(1\!-\!\sigma)m\ell\cosh(m\ell)]\, F_3^+(\eps)\big\}\, ,\\
\notag  &G_4 :=\tfrac{a}{2}\big\{(1\!-\!\sigma)m^3\cosh(m\ell)\,F_2^+(\eps)- m^2[(1\!+\!\sigma)\cosh(m\ell)-(1\!-\!\sigma)m\ell\sinh(m\ell)]\, F_3^+(\eps)\\
& +(1\!-\!\sigma)m^3\sinh(m\ell)\, F_1^-(\eps)-  m^2[(1\!+\!\sigma)\sinh(m\ell)-(1\!-\!\sigma)m\ell\cosh(m\ell)]\, F_4^-(\eps)\big\}\,.
\end{eqnarray}
}

\normalsize
Then, we set
\small{
\begin{eqnarray}\label{A2B2C2D2}
\notag &C_2:=\frac{m\cosh(m\ell)\left(K_\alpha\primo_m \frac{\sigma m^2\!-\!\alpha^2}{(m^2\!-\!\alpha^2)^2} \sinh(\alpha\ell)+G_3\right)+\sinh(m\ell)\left(\alpha K_\alpha \primo_m \frac{(2\!-\!\sigma)m^2\!-\!\alpha^2}{(m^2\!-\!\alpha^2)^2} \cosh(\alpha\ell)+G_4\right)}{m^2\left[(3+\sigma)\sinh(m\ell)\cosh(m\ell)+(1-\sigma)m\ell\right]}\,,\\
\notag &D_2:=\frac{m\sinh(m\ell)\left(K_\alpha\primo_m \frac{\sigma m^2\!-\!\alpha^2}{(m^2\!-\!\alpha^2)^2} \cosh(\alpha\ell)+G_1\right)+\cosh(m\ell)\left(\alpha K_\alpha \primo_m \frac{(2\!-\!\sigma)m^2\!-\!\alpha^2}{(m^2\!-\!\alpha^2)^2} \sinh(\alpha\ell)+G_2\right)}{m^2\left[ (3+\sigma)\sinh(m\ell)\cosh(m\ell)-(1-\sigma)m\ell\right]}\,,\\
\notag  &A_2:=\frac{D_2 m^2[(1+\sigma)\sinh(m\ell)-(1-\sigma)m\ell\cosh(m\ell)]-\alpha K_\alpha \primo_m \frac{(2\!-\!\sigma)m^2\!-\!\alpha^2}{(m^2\!-\!\alpha^2)^2} \sinh(\alpha\ell)-G_2}{(1\!-\!\sigma)m^3\sinh(m\ell)}\,, \\
&B_2:=\frac{C_2m^2[(1\!+\!\sigma)\cosh(m\ell)-(1\!-\!\sigma)m\ell\sinh(m\ell)]-\alpha K_\alpha \primo_m \frac{(2\!-\!\sigma)m^2\!-\!\alpha^2}{(m^2\!-\!\alpha^2)^2} \cosh(\alpha\ell)-G_4 }{(1\!-\!\sigma)m^3\cosh(m\ell)}\,,
\end{eqnarray}}
\normalsize
and
{\small
\begin{equation}\label{A1D3}
\begin{split}
A_1&:=A_2+aF_1(\eps)\,, \quad A_3:=A_2+a e^{-2\alpha\varepsilon} F_1(-\eps)\,, \\
B_1&:=B_2+aF_2(\eps) \,, \quad B_3:=B_2+a e^{-2\alpha\varepsilon}F_2(-\eps)\,,\\
C_1&:=C_2+aF_3(\eps)\,, \quad C_3:=C_2+a e^{-2\alpha\varepsilon}F_3(-\eps)\,, \\
D_1&:=D_2+aF_4(\eps)\,, \quad D_3:=D_2+a e^{-2\alpha\varepsilon}F_4(-\eps)\,.
\end{split}
\end{equation}}

\normalsize
The following statement allows us to determine the explicit solution of \eq{eq:Poisson_CROSS}.

\begin{proposition}\label{exponentialcross}
Assume that $g\in L^2(]0,\pi[)$ satisfies \eqref{serieg}. For $\alpha\ge0$ with $\alpha\not\in\mathbb{N}$, let $K_\alpha$ be as in \eqref{f2}. Then, the unique solution
of \eqref{eq:Poisson_CROSS} is given by
$$w_{\alpha}(x,y)=\begin{cases}
 w_1(x,y)\,, & \qquad \text{in } ]0,\pi[\times ]\varepsilon,\ell[\,, \\
  w_2(x,y)\,, & \qquad \text{in } ]0,\pi[\times ]-\varepsilon,\varepsilon[\,, \\
   w_3(x,y)\,, & \qquad \text{in } ]0,\pi[\times ]-\ell,-\varepsilon[ \,.
\end{cases}
$$
with, for $i=1,\dots, 3$,
\begin{equation*}
\begin{split}
w_i(x,y)\!:=&\!\!\serie \!\! \big((A_i+C_iy)\cosh(my)\!+\!(B_i+D_iy)\sinh(my)\!+\!\tfrac{e^{\alpha y}K_\alpha\gamma_m^i}{(m^2-\alpha^2)^2}\big)\!\smx\\
\end{split}
\end{equation*}
and the constants $A_i=A_i(m,\alpha,\eps)$, $B_i=B_i(m,\alpha,\eps)$, $C_i=C_i(m,\alpha,\eps)$, and $D_i=D_i(m,\alpha,\eps)$ as defined in \eqref{A2B2C2D2} and \eqref{A1D3}, while $\gamma_m^1=\gamma_m^3=\gamma_m$ and $\gamma_m^2=\widehat \gamma_m$.
\end{proposition}

\begin{proof}
In order to solve the problem, we split the domain $\Omega$ into three rectangles:
$R_1:=]0,\pi[\times ]\varepsilon,\ell[$, $R_2:=]0,\pi[\times ]-\varepsilon,\varepsilon[$,
$R_3:=]0,\pi[\times]-\ell,-\varepsilon[$,
so that we obtain
\begin{equation} \label{eq:Poisson1}
\small
\begin{cases}
\Delta^2 w_1=K_\alpha e^{\alpha y}g(x)(1+d\chi_{I^N}(x))^{-1}\,,  &\qquad  \text{in } R_1\,, \\
w_1=(w_1)_{xx}=0\,,  &\qquad\text{on } \{0,\pi\}\times ]\varepsilon,\ell[\,, \\
(w_1)_{yy}+\sigma (w_1)_{xx}=(w_1)_{yyy}+(2-\sigma)(w_1)_{xxy}=0\,,  & \qquad\text{on } ]0,\pi[\times \{\ell\} \, ,\\
\end{cases}
\end{equation}
\begin{equation} \label{eq:Poisson2}
\small
\begin{cases}
\Delta^2 w_2=K_\alpha e^{\alpha y}g(x)(1+d)^{-1}\,, & \qquad \text{in } R_2\,, \\
w_2=(w_2)_{xx}=0\,,  &\qquad \text{on } \{0,\pi\}\times ]-\varepsilon,\varepsilon[ \, , \\
\end{cases}
\end{equation}
\begin{equation} \label{eq:Poisson3}
\small
\begin{cases}
\Delta^2 w_3=K_\alpha e^{\alpha y}g(x)(1+d\chi_{I^N}(x))^{-1}\,, & \qquad \text{in } R_3\,, \\
w_3=(w_3)_{xx}=0\,,  &\qquad\text{on } \{0,\pi\}\times ]-\ell,\varepsilon[\,, \\
(w_3)_{yy}+\sigma (w_3)_{xx}=(w_3)_{yyy}+(2-\sigma)(w_3)_{xxy}=0\,,  & \qquad\text{on } ]0,\pi[\times \{-\ell\}\, ,\\
\end{cases}
\end{equation}
where $I^N$ is as defined in \eq{INN}. We also have to add the junction conditions:
\begin{equation} \label{junction12}
\small
w_1=w_2\, ,\ \ (w_1)_{y}=(w_2)_{y}\, ,\ \ (w_1)_{yy}=(w_2)_{yy}\, ,\ \ (w_1)_{yyy}=(w_2)_{yyy}\, ,\ \ {\rm in\ } ]0,\pi[ \times\{\varepsilon\}
\end{equation}
\begin{equation} \label{junction23}
\small
w_2=w_3\, ,\ \ (w_2)_{y}=(w_3)_{y}\, ,\ \ (w_2)_{yy}=(w_3)_{yy}\, ,\ \ (w_2)_{yyy}=(w_3)_{yyy}\, ,\ \ {\rm in\ } ]0,\pi[ \times\{-\varepsilon\}
\end{equation}

Let $\primop\in H^4(]0,\pi[)$ be the unique solution of
\neweq{equtildephi}
\left\{\begin{array}{ll}
\primop''''(x)+2\alpha^2\primop''(x)+\alpha^4\primop(x)=g(x)(1+d\chi_{I^N}(x))^{-1}\,, \qquad x\in]0,\pi[\,,\\
\primop(0)=\primop(\pi)=\primop''(0)=\primop''(\pi)=0\, .
\end{array}\right.
\endeq
By \eq{serietilde}, and recalling that $\alpha\not\in\mathbb{N}$, $\primop$ may be written as
$$
\primop(x)=\serie\frac{\primo_m}{(m^2-\alpha^2)^2}\smx\,, \quad x\in]0,\pi[\,,
$$
with the $\primo_m$ as defined in \eqref{gtilde}. Moreover, $\primop''\in H^2(]0,\pi[)$ is given by
$$
\primop''(x)=-\serie\frac{\primo_m\, m^2}{(m^2-\alpha^2)^2}\smx\,, \quad x\in]0,\pi[\,,
$$
and this series converges in $H^2(]0,\pi[)$ and, hence, uniformly. We will also need the constants
\begin{equation*}
\zeta_m^1:=\primo_m\, \frac{\sigma m^2\!-\!\alpha^2}{(m^2\!-\!\alpha^2)^2}\,,\qquad \zeta_m^2:=\primo_m\, \frac{(2\!-\!\sigma)m^2\!-\!\alpha^2}{(m^2\!-\!\alpha^2)^2}\,.
\end{equation*}

Let us now restrict our attention to $R_1$. By the system \eqref{equtildephi}, we have that $\Delta^2[K_\alpha e^{\alpha y}\primop(x)]=K_\alpha e^{\alpha y}g(x)(1+d\chi_{I^N}(x))^{-1}$. Hence, if we introduce the auxiliary function
$v_1(x,y):=w_1(x,y)-K_\alpha e^{\alpha y}\primop(x)$ with $w_1$ solving \eqref{eq:Poisson1}, we see that $v_1$ solves
\neweq{pbv1}
\small
\left\{\begin{array}{ll}
\Delta^2v_1=0\,,\quad & \mbox{in }R_1\,,\\
v_1=(v_1)_{xx}=0\,,\quad & \mbox{on }\{0,\pi\}\times ]\varepsilon,\ell[\,,\\
(v_1)_{yy}+\sigma (v_1)_{xx}=-K_\alpha e^{\alpha\ell}[\alpha^2\primop+\sigma\primop'']\,,\quad & \mbox{on }]0,\pi[\times\{\ell\}\,,\\
(v_1)_{yyy}+(2-\sigma)(v_1)_{xxy}=-K_\alpha \alpha e^{\alpha\ell}[\alpha^2\primop+(2-\sigma)\primop'']\,,  & \mbox{on }]0,\pi[\times\{\ell\}\, .
\end{array}\right.
\endeq
We seek solutions of \eqref{pbv1} by separating variables, namely we seek functions $Y_m^1=Y_m^1(y)$ such that
$v_1(x,y)=\serie Y_m^1(y)\smx$
solves \eqref{pbv1}. Then
$$\Delta^2v_1(x,y)=\serie[(Y_m^1)''''(y)-2m^2(Y_m^1)''(y)+m^4(Y_m^1)(y)]\smx\,,$$
and the equation in \eqref{pbv1} yields
\neweq{Ym}
(Y_m^1)''''(y)-2m^2(Y_m^1)''(y)+m^4Y_m^1(y)=0\,,\quad\mbox{for }y\in ]\varepsilon,\ell[\, .
\endeq
The solutions of \eqref{Ym} are linear combinations of $\cosh(my)$, $\sinh(my)$, $y\cosh(my)$, $y\sinh(my)$, that is,
$Y_m^1(y)=(A_1+C_1y)\cosh(my)+(B_1+D_1y)\sinh(my)$, 
where $A_1$, $B_1$, $C_1$, $D_1$ will be determined by imposing the boundary conditions in \eqref{junction12} and \eqref{pbv1}.
By differentiating we obtain
{\small
\begin{equation*}
\begin{split}
(Y_m^1)'(y)&=(A_1m+D_1+C_1my)\sinh(my)+(B_1m+C_1+D_1my)\cosh(my)+\sinh(my)\, ,\\
(Y_m^1)''(y)&=(A_1m^2+2D_1m+C_1m^2y)\cosh(my)+(B_1m^2+2C_1m+D_1m^2y)\sinh(my)\, ,\\
(Y_m^1)'''(y)&=(A_1m^3+3D_1m^2+C_1m^3y)\sinh(my)+(B_1m^3+3C_1m^2+D_1m^3y)\cosh(my)\, .
\end{split}
\end{equation*}}
The two boundary conditions on $]0,\pi[\times\{\varepsilon,\ell\}$, see \eqref{pbv1}, become respectively
{\small \begin{equation*}
\begin{split}
\serie[(Y_m^1)''(\ell)-\sigma m^2 Y_m^1(\ell)]\smx &=- K_\alpha e^{\alpha\ell}[\alpha^2\primop(x)+\sigma\primop''(x)]\, ,\\
\serie[(Y_m^1)'''(\ell)-(2-\sigma)m^2(Y_m^1)'(\ell)]\smx &=-K_\alpha e^{\alpha\ell}[\alpha^2\primop(x)+(2-\sigma)\primop''(x)]\, ,
\end{split}
\end{equation*}}
for all $x\in]0,\pi[$. Hence, from the Fourier expansion of $\phi$, we deduce that
$(Y_m^1)''(\ell)-\sigma m^2Y_m^1(\ell)=K_\alpha e^{\alpha\ell}\, \zeta_m^1$,
$(Y_m^1)'''(\ell)-(2\!-\!\sigma)m^2(Y_m^1)'(\ell)=K_\alpha e^{\alpha\ell}\, \alpha\, \zeta_m^2$.
By plugging this information into the explicit form of the derivatives of $Y_m^1$ we find the system
{\small
\neweq{system1}
\begin{array}{l}
\notag (1\!-\!\sigma)m^2\cosh(m\ell)\, A_1+m[2\cosh(m\ell)+(1\!-\!\sigma)m\ell\sinh(m\ell)]\, D_1+(1\!-\!\sigma)m^2\sinh(m\ell)\, B_1  \\
\notag + m[2\sinh(m\ell)+(1\!-\!\sigma)m\ell\cosh(m\ell)]\, C_1  = K_\alpha e^{\alpha\ell}\zeta_m^1\,,\\
\notag -(1\!-\!\sigma)m^3\cosh(m\ell)\, B_1+m^2[(1\!+\!\sigma)\cosh(m\ell)-(1\!-\!\sigma)m\ell\sinh(m\ell)]\, C_1\\ 
-(1\!-\!\sigma)m^3\sinh(m\ell)\, A_1 + m^2[(1\!+\!\sigma)\sinh(m\ell)-(1\!-\!\sigma)m\ell\cosh(m\ell)]\, D_1  =
K_\alpha e^{\alpha\ell}\, \alpha\, \zeta_m^2\,.\\
\end{array}
\endeq}

Similarly, for $R_3$ we introduce the function
$v_3(x,y):=w_3(x,y)-K_\alpha e^{\alpha y}\primop(x)$, with $w_3$ solving \eqref{eq:Poisson3}, and we see that $v_3$ satisfies
\neweq{pbv3}
\small
\left\{\begin{array}{ll}
\Delta^2v_3=0\,,\quad & \mbox{in }R_3\,,\\
v_3=(v_3)_{xx}=0\,,\quad & \mbox{on }\{0,\pi\}\times ]-\ell,-\varepsilon[\,,\\
(v_3)_{yy}+\sigma (v_3)_{xx}=-K_\alpha e^{-\alpha\ell}[\alpha^2\primop+\sigma\primop'']\,,\quad & \mbox{on }]0,\pi[\times\{-\ell\}\,,\\
(v_3)_{yyy}+(2-\sigma)(v_3)_{xxy}=-K_\alpha \alpha e^{-\alpha\ell}[\alpha^2\primop+(2-\sigma)\primop'']\,,  & \mbox{on }]0,\pi[\times\{-\ell\}\, .
\end{array}\right.
\endeq
By separating variables, we seek functions $Y_m^3=Y_m^3(y)$ so that the function
$v_3(x,y)=\serie Y_m^3(y)\smx$
solves \eqref{pbv3}. Then
$$\Delta^2v_3(x,y)=\serie[(Y_m^3)''''(y)-2m^2(Y_m^3)''(yY_m^2)+m^4Y_m^3(y)]\smx$$
and then
$Y_m^3(y)=A_3\cosh(my)+B_3\sinh(my)+C_3y\cosh(my)+D_3y\sinh(my)$.
As with $R_1$, we are then led to the system
{\small
\neweq{system2}
\begin{array}{l}
\notag (1\!-\!\sigma)m^2\cosh(m\ell)\, A_3+m[2\cosh(m\ell)+(1\!-\!\sigma)m\ell\sinh(m\ell)]\, D_3   \\
\notag -(1\!-\!\sigma)m^2\sinh(m\ell)\, B_3- m[2\sinh(m\ell)+(1\!-\!\sigma)m\ell\cosh(m\ell)]\, C_3  =
K_\alpha e^{-\alpha\ell}\, \zeta_m^1\,, \\
\notag -(1\!-\!\sigma)m^3\cosh(m\ell)\, B_3+m^2[(1\!+\!\sigma)\cosh(m\ell)-(1\!-\!\sigma)m\ell\sinh(m\ell)]\, C_3 \\
 +(1\!-\!\sigma)m^3\sinh(m\ell)\, A_3 -m^2[(1\!+\!\sigma)\sinh(m\ell)-(1\!-\!\sigma)m\ell\cosh(m\ell)]\, D_3  =
 K_\alpha e^{-\alpha\ell}\, \alpha\, \zeta_m^2\, .
\end{array}
\endeq}

Finally, let $\secondo_m$ be the Fourier coefficients of $g$, see \eqref{serieg}, and $\secondop\in H^4(]0,\pi[)$ be defined as $
\secondop(x)=\serie\frac{\secondo_m}{(m^2-\alpha^2)^2}\smx$, for every $x\in]0,\pi[$.
For $R_2$ we introduce the auxiliary function
$v_2(x,y):=w_2(x,y)-K_\alpha e^{\alpha y}\frac{\secondop(x)}{1+d}$ with $\secondop$ as above and $w_2$ solving \eqref{eq:Poisson2}, and we see that $v_2$ satisfies
\neweq{pbv2}
\left\{\begin{array}{ll}
\Delta^2v_2=0\quad & \mbox{in }R_2\\
v_2=(v_2)_{xx}=0\quad & \mbox{on }\{0,\pi\}\times ]-\varepsilon,\varepsilon[\, .
\end{array}\right.
\endeq
We seek again solutions of \eqref{pbv2} by separating variables, namely we seek functions $Y_m^2=Y_m^2(y)$ such that
$v_2(x,y)=\serie Y_m^2(y)\smx$
solves \eqref{pbv2}. Then
$$\Delta^2v_2(x,y)=\serie[(Y_m^2)''''(y)-2m^2(Y_m^2)''(y)+m^4Y_m^2(y)]\smx$$
and then
$
Y_m^2(y)=A_2\cosh(my)+B_2\sinh(my)+C_2y\cosh(my)+D_2y\sinh(my)\,.
$

In this case we have no boundary conditions to use as a constraint. Instead, we impose the junction conditions \eqref{junction12} and \eqref{junction23} by which we get the relations \eqref{A1D3}.
Combining \eqref{A1D3} with \eqref{system1} and \eqref{system2}, we obtain a $4\times 4$ system in the unknowns $A_2, B_2, C_2, D_2$ which decouples into the following $2\times 2$ systems:
 {\small $$
\left\{\begin{array}{ll}
\!\!\!\!(1\!-\!\sigma)m^2\cosh(m\ell) A_2+m[2\cosh(m\ell)+(1\!-\!\sigma)m\ell\sinh(m\ell)] D_2\!=\!
K_\alpha \zeta_m^1 \cosh(\alpha\ell)+G_1\\
\!\!\!\!-(1\!-\!\sigma)m^3\!\sinh(m\ell) A_2\!\!+\!m^2[(1\!+\!\sigma)\!\sinh(m\ell)\!\!-\!\!(1\!-\!\sigma)m\ell\!\cosh(m\ell)] D_2\!\!=\!
\alpha K_\alpha \zeta_m^2 \sinh(\alpha\ell)\!\!+\!G_2 
\end{array}\right.
$$}
and
{\small $$
\left\{\begin{array}{ll}
\!\!\!\!(1\!-\!\sigma)m^2\sinh(m\ell) B_2\!\!+\!m[2\!\sinh(m\ell)+(1\!-\!\sigma)m\ell\!\cosh(m\ell)] C_2\!\!=\!
K_\alpha \zeta_m^1\! \sinh(\alpha\ell)\!\!+\!G_3\\
\!\!\!\!-(1\!-\!\sigma)m^3\!\cosh(m\ell) B_2\!\!+\!m^2[(1\!+\!\sigma)\!\cosh(m\ell)\!\!-\!\!(1\!-\!\sigma)m\ell\!\sinh(m\ell)] C_2\!\!=\!
\alpha K_\alpha\zeta_m^2 \cosh(\alpha\ell)\!\!+\!G_4
\end{array}\right.
$$}
where the $G_i=G_i(m,\alpha,\eps)$ are defined in \eqref{G}. The solutions of the above systems are given in  \eqref{A2B2C2D2} and, combined with \eq{A1D3}, allow us to write the explicit form of $w_1$, $w_2$ and $w_3$.\par
To complete the proof of Proposition \ref{exponentialcross} we show that the series defining $w_\alpha$ converges in the next lemma.\end{proof}

\begin{lemma}
If $g\in L^2(]0,\pi[)$, then the series defining $w_\alpha$ in Proposition \ref{exponentialcross} converges uniformly in $\overline \Omega$ up to the second derivative.
\end{lemma}
\begin{proof}
We start by studying the uniform convergence of the series which defines $w_1$. We have
{\small
$$|w_1|\leq\!\serie \!\Big( \frac{e^{m\ell}}{2} \left(|A_1+B_1|+\ell |C_1+D_1|\right)+\frac{e^{-m\eps}}{2} \left(|A_1-B_1|+\ell |C_1-D_1|\right)+\tfrac{e^{\alpha \ell}K_\alpha\primo_m}{(m^2-\alpha^2)^2}\Big).$$}
From the following relations
\begin{equation*}
\small
\begin{split}
\frac{2 m\tanh(m\ell)G_1}{a}+\frac{2 G_2}{a}=& \displaystyle \frac{1}{\cosh(m\ell)}[(1-\sigma)m^3F_2^-+2m^2F_3^-+m^3(1-\sigma)\ell F_4^+)]\\
&-(\sigma+3)m^2[F_3^-\cosh(m\ell)+F_4^+ \sinh(m\ell)]\,,\\
\frac{2 \tanh(m\ell)G_4}{a}+\frac{2m G_3}{a}=& \displaystyle \frac{1}{\cosh(m\ell)}[-(1-\sigma)m^3F_1^--(1-\sigma)m^3\ell F_3^++m^2(1+\sigma) F_4^-)]\\
&-(\sigma+3)m^2[F_4^-\cosh(m\ell)+ F_3^+ \sinh(m\ell)]\,,
\end{split}
\end{equation*}
and after some lengthy calculations, as $m \rightarrow +\infty$, we get
\begin{equation}\label{A2+B2_C2+D2}
\begin{split}
A_2+B_2& = -a(F_1(\eps)+F_2(\eps))+\overline K_1 \frac{\primo_m e^{-m\ell}}{m^3}+o\left(\frac{\primo_m e^{-m\ell}}{m^3}\right)\,,\\
C_2+D_2& = -a(F_3(\eps)+F_4(\eps))+ \overline K_2 \frac{\primo_m e^{-m\ell}}{m^3}+o\left(\frac{\primo_m e^{-m\ell}}{m^3}\right)\,,
\end{split}
\end{equation}
for some $\overline K_1, \overline K_2 \in \R \setminus \{0\}$. Hence, as $m\to\infty$, we have
\begin{equation}\label{A1+B1_C1+D1}
\small
A_1+B_1=\overline K_1 \frac{\primo_m e^{-m\ell}}{m^3}+o\left(\tfrac{\primo_m e^{-m\ell}}{m^3}\right)\,,\quad
C_1+D_1=\overline K_2 \frac{\primo_m e^{-m\ell}}{m^3}+o\left(\tfrac{\primo_m e^{-m\ell}}{m^3}\right)\,.
\end{equation}
On the other hand, as $m \rightarrow +\infty$, we have
\begin{equation}\label{A2-B2_C2-D2}
\begin{split}
A_2-B_2& =  \overline K_3 (\primo_m-\widehat \gamma_m) \frac{e^{-m\eps}}{m^3}+o\left((\primo_m-\widehat \gamma_m)\frac{e^{-m\eps}}{m^3}\right)\,,\\
C_2-D_2& =  \overline K_4 (\primo_m-\widehat \gamma_m) \frac{ e^{-m\eps}}{m^3}+o\left((\primo_m-\widehat \gamma_m) \frac{e^{-m\eps}}{m^3}\right)\,,
\end{split}
\end{equation}
for some $\overline K_3, \overline K_4 \in \R \setminus \{0\}$. Hence, since as $m \rightarrow +\infty$ we have
\begin{equation*}
\begin{split}
a(F_1(\eps)-F_2(\eps))&=- K_\alpha \eps e^{\alpha \eps}(\primo_m-\widehat \gamma_m) \frac{ e^{m\eps}}{2m^3}+o\left((\primo_m-\widehat \gamma_m) \frac{ e^{m\eps}}{ m^3}\right) \,,\\
 a(F_3(\eps)-F_4(\eps))&= K_\alpha e^{\alpha \eps}(\primo_m-\widehat \gamma_m) \frac{ e^{m\eps}}{2m^3}+o\left((\primo_m-\widehat \gamma_m) \frac{ e^{m\eps}}{ m^3}\right)\,,
\end{split}
\end{equation*}
we conclude that
\begin{eqnarray*}
A_1-B_1& = & -K_\alpha  \eps e^{\alpha \eps} (\primo_m-\widehat \gamma_m)  \frac{e^{m\eps}}{2m^3}+o\left((\primo_m-\widehat \gamma_m)\frac{e^{m\eps}}{m^3}\right)\,,\\
C_1-D_1& = & K_\alpha  e^{\alpha \eps} (\primo_m-\widehat \gamma_m) \frac{ e^{m\eps}}{2m^3}+o\left((\primo_m-\widehat \gamma_m) \frac{e^{m\eps}}{m^3}\right)\,.
\end{eqnarray*}
The above equalities, together with \eqref{A1+B1_C1+D1} and the summability of the coefficients $ \secondo_m$ and $\widehat \gamma_m$, prove the uniform convergence of
the series defining $w_1$ up to the second derivative.\par
Next,we consider $w_2$. We estimate
{\small
$$|w_2|\leq\!\serie \!\Big( \frac{e^{m\eps}}{2} \left(|A_2+B_2|+\eps |C_2+D_2|\right)+\frac{e^{m\eps}}{2} \left(|A_2-B_2|+\eps |C_2-D_2|\right)+\frac{e^{\alpha \eps}K_\alpha\widehat \gamma_m}{(m^2-\alpha^2)^2}\Big).$$}
Since as $m \rightarrow +\infty$ we have
$$a(F_1(\eps)+F_2(\eps))= K_\alpha \eps e^{\alpha \eps}(\primo_m-\widehat \gamma_m) \frac{ e^{-m\eps}}{2 m^3}+o\left((\primo_m-\widehat \gamma_m) \frac{ e^{-m\eps}}{ m^3}\right) \,,$$
$$ a(F_3(\eps)+F_4(\eps))= K_\alpha e^{\alpha \eps}(\primo_m-\widehat \gamma_m) \frac{ e^{-m\eps}}{2m^3}+o\left((\primo_m-\widehat \gamma_m) \frac{ e^{-m\eps}}{ m^3}\right)\,,$$
from \eqref{A2+B2_C2+D2} we deduce that
\begin{eqnarray*}
A_2+B_2& = &-K_\alpha \eps e^{\alpha \eps}(\primo_m-\widehat \gamma_m) \frac{ e^{-m\eps}}{2 m^3}+o\left((\primo_m-\widehat \gamma_m) \frac{ e^{-m\eps}}{ m^3}\right)\,,\notag\\
C_2+D_2& = & - K_\alpha e^{\alpha \eps}(\primo_m-\widehat \gamma_m) \frac{ e^{-m\eps}}{2m^3}+o\left((\primo_m-\widehat \gamma_m) \frac{ e^{-m\eps}}{ m^3}\right)\,.
\end{eqnarray*}
This, jointly with \eqref{A2-B2_C2-D2} and the summability of the coefficients $ \secondo_m,\, \widehat \gamma_m$, proves the uniform convergence of the series defining $w_2$ up to the second derivative.\par
The computations for $w_3$ are similar to those for $w_1$ and we omit them.\end{proof}


Now, we focus on the limiting behavior of $w_{\alpha}$ as $\alpha \rightarrow +\infty$.
We first determine the limit of $h_\alpha$.

\begin{lemma}\label{limitfalpha_rinforzo}
Let $C_g$ be as in \eqref{f2} and $I^N$ be as in \eqref{INN}.
As $\alpha\to+\infty$ we have that $h_{\alpha} \to \overline h $ in $H^{-2}_*(\Omega)$, where $\overline h \in H^{-2}_*(\Omega)$ is defined as follows
$$\langle \overline h,v\rangle=\int_0^\pi \frac{g(x)}{C_g(1+d\chi_{I^N}(x))}\, v(x,\ell)\, dx\qquad\forall v\in H^2_*(\Omega)\,.$$
\end{lemma}

The proof is a consequence of an integration by parts, similar to that of Lemma \ref{limitfalphanew} below and therefore we omit it.
Next, we set
\begin{equation}\label{ubar}
\overline w (x,y)=\!\!\serie\!\! \left((\overline A_2+\overline C_2y)\cosh(my)+(\overline B_2+\overline D_2y)\sinh(my)\right)\smx\,,
\end{equation}
where $\overline A_2=\overline A_2(m)$, $\overline B_2=\overline B_2(m)$, $\overline C_2=\overline C_2(m)$, $\overline D_2=\overline D_2(m)$ are given by
\begin{equation*}
\begin{split}
\overline A_2:=& \frac{\primo_m \left((1-\sigma)m\ell\sinh(m\ell)+2\cosh(m\ell)\right)}{2C_g(1\!-\!\sigma)m^3\left[ (3+\sigma)\sinh(m\ell)\cosh(m\ell)-(1-\sigma)m\ell\right]}\,,\\
\overline B_2:=&  \frac{ \primo_m \left((1-\sigma)m\ell\cosh(m\ell)+2\sinh(m\ell)\right)}{2C_g (1\!-\!\sigma)m^3\left[(3+\sigma)\sinh(m\ell)\cosh(m\ell)+(1-\sigma)m\ell\right]}\,,\\
\overline C_2:=& - \frac{\primo_m\sinh(m\ell) }{2C_g m^2\left[(3+\sigma)\sinh(m\ell)\cosh(m\ell)+(1-\sigma)m\ell \right]}\,,\\
\overline D_2:=& - \frac{\primo_m \cosh(m\ell)}{2C_g m^2\left[ (3+\sigma)\sinh(m\ell)\cosh(m\ell)-(1-\sigma)m\ell\right]}\,,
\end{split}
\end{equation*}
with $\primo_m$ as in \eq{gtilde}. The following lemma holds.
\begin{lemma}\label{limitfalpha2_rinforzo}
Let $w_\alpha$ be the unique solution of \eqref{eq:Poisson_CROSS}, $\overline w$ be as in \eqref{ubar} and $\overline h$ be as in Lemma \ref{limitfalpha_rinforzo}. Then, as $\alpha\to+\infty$, we have that $w_\alpha(x,y)\to\overline w\left(x,y\right)$ in $H_*^2(\Omega)$ and $\overline w$ is the unique
solution of the problem
\neweq{weaklimit}
(w,v)_{H^2_*}=\langle \overline h,v\rangle\qquad\forall v\in H^2_*(\Omega)\,.
\endeq
\end{lemma}
\begin{proof}
Let $A_2$, $B_2$, $C_2$, $D_2$ be as defined in \eqref{A2B2C2D2}, let $\overline A_2$, $\overline B_2$, $\overline C_2$, $\overline D_2$ be as defined in \eqref{ubar}, we get $A_2\to \overline A_2$, $B_2\to \overline B_2$, $C_2\to \overline C_2$, $D_2\to \overline D_2$ {as }$\alpha \rightarrow +\infty$.
Moreover, from \eqref{f2}, \eqref{Fi}, and \eqref{a}, as $\alpha\to+\infty$, we have
$$K_\alpha=\frac{\alpha e^{- \alpha \ell}}{ C_g}+o(\alpha e^{- \alpha \ell} )\,,\qquad   a= \frac{e^{-\alpha(\ell-\eps)}(\primo_m-\widehat\gamma_m)}{C_g \alpha^3} +o\left(\frac{e^{-\alpha(\ell-\eps)}}{ \alpha^3}\right)\,,$$
and
\begin{eqnarray*}
F_1(\eps)=\frac{\sinh(m\varepsilon)-m\varepsilon\cosh(m\varepsilon)}{2m^3} \alpha^3+o(\alpha^3)\,, & F_3(\eps) = \frac{\cosh(m\varepsilon)}{2m^2}\, \alpha^3+o(\alpha^3)  \,,\\
F_2(\eps)= \frac{m\varepsilon\sinh(m\varepsilon)-\cosh(m\varepsilon)}{2m^3}\, \alpha^3+o(\alpha^3)\,, & F_4(\eps) = -\frac{\sinh(m\varepsilon)}{2m^2}\,  \alpha^3+o(\alpha^3)\, .
\end{eqnarray*}
Hence, recalling the definition of the constants $A_i$, $B_i$, $C_i$, $D_i$ given in \eqref{A2B2C2D2} and \eqref{A1D3}, for $i=1,2,3$, we deduce that $A_i\to \overline A_2$, $B_i\to \overline B_2$, $C_i\to \overline C_2$, $D_i\to \overline D_2$, as $\alpha \rightarrow +\infty$.
Therefore, by exploiting the summability of the coefficients in the series defining $w_\alpha$, we conclude that $w_\alpha\to \overline w $ a.e.\ in $\Omega$, as $\alpha\to+\infty$. \par
On the other hand, by Lemma \ref{limitfalpha_rinforzo}, there holds $h_{\alpha} \to \overline h $ in $H^{-2}_*(\Omega)$ as $\alpha\to+\infty$. Let $w_{\overline h} \in H^2_*(\Omega)$ be the unique solution of \eqref{weaklimit}, so that clearly
$\|w_{\alpha}- w_{\overline h} \|_{H^2_*}\leq \|h_{\alpha}-\overline h\|_{H^{-2}_*}$, $w_{\alpha}\to  w_{\overline h} $ in $H^2_*(\Omega)$ and uniformly in $\overline \Omega$. Hence, $\overline w  \equiv w_{\overline h}$ and this concludes the proof of the lemma.\end{proof}

\begin{proof}[Proof of Theorem \ref{gap_limit_swiss}]
By Proposition \ref{exponentialcross}, the gap function corresponding to $w_\alpha$ is
$\GG_\alpha\left(x\right)=w_1(x,\ell)-w_3(x,-\ell)=\serie \xi_m(\alpha) \smx$
with the coefficients
\begin{eqnarray}\label{Im}
\xi_m(\alpha)&:=& (A_1-A_3)\cosh(m\ell)+(B_1+B_3)\sinh(m\ell)+(C_1+C_3)\ell\cosh(m\ell)\notag\\
&\ & +(D_1-D_3)\ell\sinh(m\ell)+\frac{\alpha \primo_m}{C_g(m^2-\alpha^2)^2}\,,
\end{eqnarray}
while the gap function corresponding to $\overline w$ can be written as the function
$\overline \GG\left(x\right)=\overline w\left(x,\ell\right)-\overline w\left(x,-\ell\right)=\serie \Emo \smx$
with the coefficients
\begin{eqnarray}\label{Em_asym}
\Emo&=& \left(2\overline B_2\sinh(m\ell)+2\overline C_2\ell\cosh(m\ell)\right)\smx=\frac{2\, \primo_m \, \Upsilon_m}{C_g(1-\sigma)}\,,
\end{eqnarray}
where the coefficients $\Upsilon_m$ are as defined in \eq{Um}.
Since, by Lemma \ref{limitfalpha2_rinforzo}, the function $w_\alpha$ converges to $\overline w$ uniformly in $\overline \Omega$, we have
\begin{equation*}
\begin{split}
\max_{x\in[0,\pi]}|\GG_\alpha\left(x\right)-\overline \GG \left(x\right)|=\max_{x\in[0,\pi]}\left|\serie \xi_m(\alpha)\smx-\serie \Emo \smx \right| \\
\notag \leq \max_{x\in [0,\pi]}|w_\alpha\left(x,\ell\right)-\overline w\left(x,\ell\right)|+\max_{x\in [0,\pi]}|w_\alpha\left(x,-\ell\right)-\overline w\left(x,-\ell\right)|,
\end{split}
\end{equation*}
and so the right-hand side converges to zero.

Next, we specify the asymptotic behavior of $\xi_m(\alpha)$. To this end, by
Proposition \ref{exponentialcross}, we note that, as $\alpha \to +\infty$, the following estimates hold
{\small \begin{equation*}
\begin{split}
A_1-A_3&=a(F_1(\varepsilon)-e^{-2\alpha\varepsilon}F_1(-\varepsilon))=O(e^{-\alpha(\ell-\varepsilon)})+O(e^{-\alpha(\ell+\varepsilon)})\,,\\
B_1+B_3&=2B_2+a(F_2(\varepsilon)+e^{-2\alpha\varepsilon}F_2(-\varepsilon))=2B_2+O(e^{-\alpha(\ell-\varepsilon)})+O(e^{-\alpha(\ell+\varepsilon)})\,,\\
C_1+C_3&=2C_2+a(F_3(\varepsilon)+e^{-2\alpha\varepsilon}F_3(-\varepsilon))=2C_2+O(e^{-\alpha(\ell-\varepsilon)})+O(e^{-\alpha(\ell+\varepsilon)})\,,\\
D_1-D_3&=a(F_4(\varepsilon)-e^{-2\alpha\varepsilon}F_4(-\varepsilon))=O(e^{-\alpha(\ell-\varepsilon)})+O(e^{-\alpha(\ell+\varepsilon)})\,,\\
B_2=\overline B_2 &-\frac{\primo_m \left((1+\sigma)\cosh(m\ell)-(1-\sigma)m\ell\sinh(m\ell)\right)}{2 C_g(1-\sigma)m^2\left[(3+\sigma)\sinh(m\ell)\cosh(m\ell)+(1-\sigma)m\ell\right]}\,\frac{1}{\alpha}+o\left(\frac{1}{\alpha}\right)\,,\\
C_2=\overline C_2&- \frac{\primo_m\cosh(m\ell) }{2 C_gm\left[(3+\sigma)\sinh(m\ell)\cosh(m\ell)+(1-\sigma)m\ell\right]}\,\frac{1}{\alpha}+o\left(\frac{1}{\alpha}\right)\,.
\end{split}
\end{equation*}
}Recalling the definition of $\xi_m(\alpha)$, we conclude that, as $\alpha \to +\infty$,
{\small \begin{eqnarray*}
\xi_m(\alpha)&=&2\overline B_2 \sinh(m\ell)+2\overline C_2\ell\cosh(m\ell)\\
& &-\frac{(1+\sigma)\sinh(m\ell)\cosh (m\ell)+(1-\sigma)m\ell}{ C_g(1\!-\!\sigma)m^2[(3\!+\!\sigma)\sinh(m\ell)\cosh(m\ell)+(1\!-\!\sigma)m\ell]}\,\frac{  \primo_m}{\alpha }+o\left(\frac{1}{\alpha}\right)\\
&=&\Emo-\frac{  \overline \omega_m}{\alpha }+o\left(\frac{1}{\alpha}\right)\,
\end{eqnarray*}}
with $\overline \omega_m$ as in \eqref{Emd}.
\par
In view of the discussion at the beginning of this section, the statement of Theorem \ref{gap_limit_swiss} now follows simply by setting
\begin{equation}\label{betam}
\Em(\alpha):=\frac{R_\alpha}{K_\alpha} \,\xi_m(\alpha)
\end{equation}
with $\xi_m(\alpha)$ as given in \eq{Im}.
Notice that we still denote by $\GG_\alpha(x)$ the gap functions corresponding to $u_\alpha$.
Furthermore, by Lemma \ref{limitfalpha2_rinforzo}, $u_\alpha(x,y)\to\overline u\left(x,y\right)$ in $H_*^2(\Omega)$ where $\overline u$ is the odd part of the
unique solution of problem \eq{weaklimit}. Namely, $\overline u$ is the the
unique solution of the problem
\neweq{weaklimit2}
(\overline u,v)_{H^2_*}=\langle \overline h^o,v\rangle\qquad\forall v\in H^2_*(\Omega)\,
\endeq
where $\overline h^o$ is the odd part of $\overline h$ as defined in Lemma \ref{limitfalpha_rinforzo}.
\end{proof}

\begin{proof}[Proof of Theorem \ref{examples}]
Set $\mu=0$ and $g(x)=\sin(nx)$ for some positive integer $n$, then the coefficients \eqref{gtilde} become $\primo_n=1$ while
$\primo_m=0$ if $m\neq n$. Furthermore, since $C_g=\int_0^{\pi}|\sin(nx)|\,dx=2$, by Theorem \ref{gap_limit_swiss}  the corresponding gap function is
\begin{equation}\label{boh}
\overline\GG_n(x)=\overline \Upsilon_n\sin(nx)\,,
\end{equation}
where $\overline \Upsilon_n:=\Upsilon_n/(1-\sigma)$ with the $\Upsilon_n$ as defined in \eqref{Um}. Hence, $\overline\GG^{\infty}_n=\overline\Upsilon_n$. The thesis follows by showing that $\overline\GG^{\infty}< \overline \Upsilon_1$ for every $g\in \Gamma$.
Let us consider separately the four cases in the set $\Gamma$.\par
$\bullet$ Let $g(x)=\sin(mx)$ for some positive integer $m$. By \eq{boh} we know that $\overline \GG^{\infty}_m=\overline\Upsilon_m$. Since $0<\sigma <1$, some lengthy computations reveal that $\frac{d\overline{\Upsilon}_m}{d m}<0$, hence the map $m\mapsto \overline\Upsilon_m$
is strictly decreasing and $\max_m \overline \GG^{\infty}_m=\overline\Upsilon_1$.

$\bullet$ For given $N\in \N$ and $\{\secondo_m\}_{m\in \N} \subset \ell^2(\N)$, let $g(x)=\sum_{m\geq N} \secondo_m \sin(mx)$. Since $\sup_m |\secondo_m|\leq 2C_g/\pi$ and $\serie \overline\Upsilon_m$ converges, by Theorem \ref{gap_limit_swiss} we infer the existence of $N\in \N$ sufficiently large such that
$\GG^{\infty}\leq \frac{4}{\pi}\sum_{m\geq N} \overline\Upsilon_m<\overline\Upsilon_1$.

$\bullet$  Let $m$ be a positive integer, and take $g(x)=\sin(mx)+\sin(3m x)$. By the prosthaphaeresis formulas, $g(x)=4\sin(mx)\cos^2(m x)$ and we compute
the value of $C_{g}$ in \eq{f2}. By putting $x_k=(k/m)\pi$ with $k=0,...,m$, for $m=2n$ we have
{\small $$C_g=\sum_{j=0}^{n-1}\left(\int_{x_{2j}}^{x_{2j+1}}\big(\sin(mx)+\sin(3m x)\big)\,dx
-\int_{x_{2j+1}}^{x_{2j+2}}\left(\sin(mx)+\sin(3m x)\right)\,dx\right)\,,$$}
while for $m=2n+1$ we have
{ \small
\begin{equation*}
\begin{split}
C_g&=\sum_{j=0}^{n-1}\left[\int_{x_{2j}}^{x_{2j+1}}\!\!\big(\sin(mx)+\sin(3m x)\big)\,dx-\int_{x_{2j+1}}^{x_{2j+2}} \big(\sin(mx)+\sin(3m x)\big)\,dx\right]\\&+\int_{x_{2n}}^{x_{2n+1}}\!\!\big(\sin(mx)+\sin(3m x)\big)\,dx\,.
\end{split}
\end{equation*}}
In any case, we get that $C_g=8/3$. Furthermore, by Theorem \ref{gap_limit_swiss} we have
{\small $$|\GG(x)|=\frac{ 2 }{C_g }|\left(\overline\Upsilon_m\sin (mx)+\overline\Upsilon_{3m}\sin (3mx)\right)|\leq \frac{ 3 }{4}\left( \overline\Upsilon_m+\overline\Upsilon_{3m}\right)\leq  \frac{ 3 }{4}\left( \overline\Upsilon_1+\overline\Upsilon_3\right)\,,$$
}where in the last step we exploit the monotonicity of the map $n\mapsto \overline\Upsilon_n$. Finally, some lengthy computations reveal that $3F_3 < F_1$,
hence $\GG^{\infty}\leq \overline\Upsilon_1$ holds also for $g(x)=\sin(mx)+\sin(3m x)$.\par
$\bullet$ Take $g(x)=g_N(x)=\sum_{m=1}^N  \sin((2m-1)x)$ for $N\in \N$ sufficiently large, to be fixed later. It is known that $g_N(x)=\sin^2(Nx)/\sin(x)$
for $x\in ]0,\pi[$, see \cite[p.73]{polya}. Hence,
$$C_{g_N}=\int_0^{\pi} |g_N(x)|\,dx=\int_0^{\pi} g_N(x)\,dx=2\sum_{m=1}^N\frac{1}{2m-1}\geq \log(2N-1)\,.$$
By this and Theorem \ref{gap_limit_swiss} we deduce that
$$\GG^{\infty} \leq \frac{ 2}{C_g } \sum_{m=1}^N \overline\Upsilon_m\leq \frac{ 2}{\log (2N-1) } \sum_{m=1}^N F_m(\ell)\leq \overline\Upsilon_1$$
for $N$ sufficiently large, since the last summation converges when $N\rightarrow +\infty$. This concludes the proof of Theorem \ref{examples}.\hfill$\Box$
\end{proof}

\section{Proof of Proposition \ref{GTxi}}\label{proofdeltas}

For $z\in ]0,\pi[$, $0<\eta<\min \{z,\pi-z\}$, and $\alpha>0$, we take $g(x)=\chi_{[z-\eta, z+\eta]}(x)$ in \eq{f2sinh} and we set
$f_{\alpha,\eta}(x,y):=R_{\alpha,\eta}\, \sinh(\alpha y)\,\chi_{[z-\eta, z+\eta]}(x)$ with the constant $R_{\alpha,\eta}:=\frac{\alpha}{4\eta(\cosh(\alpha\ell)-1)}$,
so that $\|f_{\alpha,\eta}\|_{L^1}=1$. Let us establish the first ingredient for the proof of Proposition \ref{GTxi}.

\begin{lemma}\label{limitfalphanew}
Let $D=\emptyset$ and let $\GG_{\alpha,\eta}$ and $\GG_{T_{z}}$ ($T_{z}$ as in \eqref{defTi}) be the gap functions corresponding to the solutions of \eqref{weakD} with $f=f_{\alpha,\eta}$ and $f=T_{z}$, respectively. As $(\alpha^{-1},\eta)\rightarrow (0,0)$ we have that
$\GG_{\alpha,\eta}(x)\rightarrow \GG_{T_{z}}(x)$ {uniformly on } $[0,\pi]$.
\end{lemma}
\begin{proof}
We first claim that $f_{\alpha,\eta}\to T_{z}$ in $\hh(\Omega)$ as $(\alpha^{-1},\eta)\rightarrow (0,0)$,
that is,
\neweq{distribnew}
\lim_{(\alpha^{-1},\eta)\rightarrow (0,0)}\int_\Omega f_{\alpha,\eta}(x,y)v(x,y)\, dxdy=\tfrac{v(z,\ell)-v(z,-\ell)}{2}, \quad\forall v\in H^2_*(\Omega)\, .
\endeq
Take $v \in H^2_*(\Omega)$ and compute
\begin{equation*}
\small
\begin{split}
&\int_\Omega f_{\alpha,\eta}(x,y)v(x,y)\, dxdy = R_{\alpha,\eta}\,\ipi \chi_{[z-\eta, z+\eta]}(x) \left(\il  \sinh(\alpha y) v(x,y)\, dy\right)\, dx\\
& =\tfrac{1}{4\,\eta(\cosh(\alpha\ell)-1)}\int_{z-\eta}^{z+\eta}\Big(\big[\cosh(\alpha y) v(x,y)\big]_{-\ell}^\ell
-\il \cosh(\alpha y)v_y(x,y)\, dy\Big)dx\, \\
&= \tfrac{\cosh(\alpha\ell)}{\cosh(\alpha\ell)-1}\frac{v(\xi_z,\ell)-v(\xi_z,-\ell)}{2}-
\tfrac{1}{2(\cosh(\alpha\ell)-1)}  \il \cosh(\alpha y)v_y(\theta_z,y)\, dy ,
\end{split}
\end{equation*}
for some $z-\eta<\xi_z,\theta_z<z+\eta$. Now we observe that, uniformly with respect to $\eta$,
$\lim_{\alpha^{-1}\rightarrow 0}\, \il  \frac{\cosh(\alpha y)}{\cosh(\alpha\ell)-1}\,v_y(\theta_z,y)\, dy=0$,
by the Lebesgue Dominated Convergence Theorem and thus \eqref{distribnew} follows. From this we infer that the corresponding solutions converge in
$H^2_*(\Omega)$ and then the proof can be completed by arguing as in Proposition \ref{trivial}.\end{proof}

In view of Lemma \ref{limitfalphanew}, the proof of Proposition \ref{GTxi} follows once we have proved the following statement.

\begin{lemma}
Assume $z\in ]0,\pi[$, $0<\eta<\min \{z,\pi-z\}$ and $\alpha \ge0$ with $\alpha \not\in\mathbb{N}$. As $(\alpha^{-1},\eta)\rightarrow (0,0)$, the gap function $\GG_{\alpha,\eta}(x)$ corresponding to the solution of \eqref{loadpb} with $f=f_{\alpha,\eta}$ and $D=\emptyset$
converges to  $\frac{4}{\pi (1-\sigma)}  \serie \Upsilon_m \sin(m z) \sin(mx)$ uniformly on $[0,\pi]$,
with the $\Upsilon_m$ as defined in \eqref{Um}.
\end{lemma}
\begin{proof} The explicit form of the gap function $\GG_{\alpha,\eta}(x)$ follows from Theorem \ref{gap_limit_swiss} by replacing $f_{\alpha}$ with $f_{\alpha,\eta}$ and $D^N$ with $\emptyset$. Namely, we assume $\mu=\varepsilon=0$ in \eq{Deps} and $g(x)=\chi_{[z-\eta, z+\eta]}(x)$ in \eq{f2sinh}, hence $C_g=2\eta$. With these specifications, by \eq{A1D3}, we have $A_1=A_3$, $B_1=B_2=B_3$, $C_1=C_2=C_3$, and $D_1=D_3$, while $\primo_m=\secondo_m$ in \eq{serieg}. Hence, by \eq{Im} and \eq{betam}, we it follows that the function
$\GG_{\alpha,\eta}(x)=\serie \Em(\alpha,\eta) \sin (mx)
$
and
$$ \Em(\alpha,\eta)=\frac{\sinh(\alpha \ell)}{\cosh(\alpha \ell)-1} \left(2B_2\sinh(m\ell)+2C_2\ell\cosh(m\ell)+\frac{\alpha \secondo_m}{2\eta(m^2-\alpha^2)^2}\right)\,,
$$
with $B_2$ and $C_2$ as in \eq{A2B2C2D2} and $\secondo_m=\frac{4}{\pi m}\sin(mz)\sin (m\eta)$. By noting that $\frac{\secondo_{m}(\eta)}{2\eta} = \frac{2 \sin(mz)}{\pi}+o(\eta)$ as $\eta \rightarrow 0$ and exploiting \eq{Em_asym}, as $(\alpha^{-1},\eta)\rightarrow (0,0)$ we obtain
\[
\begin{split}
\Em(\alpha,\eta)&=\left(1+2e^{-\alpha \ell}+o(e^{-\alpha \ell})\right)\left(\frac{2 \sin(mz)}{\pi}+o(\eta) \right)
\left(\frac{2  \Upsilon_m}{1-\sigma} +o\left(\frac{1}{\alpha^2}\right)\right)\\
&=
\frac{4\Upsilon_m\sin(mz)}{\pi(1-\sigma)}+o(1)\,,
\end{split}
\]
with $\Upsilon_m$ as defined in \eqref{Um}. This completes the proof of the lemma.\end{proof}

\section{Conclusions, Perspectives and Open Problems}\label{open}
With possible applications to the deck of a footbridge or a suspension bridge, this paper deals with the problem of minimizing the torsional displacements of
partially hinged reinforced plates. We showed that the gap function \eqref{funzionale} is extremely useful to measure the torsional instability and that it gives
hints on how to compare the torsional performances of different plates through the minimaxmax problem \eqref{minimaxmax2}, namely a robust shape optimization in the
worst case setting. The demonstrated existence results prove that the problem is well-defined and, in some cases, it also allows to find properties of the worst
force. On the other hand, some meaningful problems prove themselves to be very difficult to handle and optimal elements are hardly characterizable. This led us to
provide some conjectures that we now motivate in detail.

There are other classes $\D$ where the minimum problem \eqref{minimaxmax2} admits a solution but the ones in Definition \ref{spazi} appear particularly
appropriate for engineering applications. A quite general class of admissible domains, where it is possible to define problem \eqref{minimaxmax2}, is that of
measurable sets with uniformly bounded De Giorgi perimeter and fixed area. In this setting, thanks to a compactness result for BV functions, the existence of a solution is still guaranteed. Therefore, we point out that it could be interesting to know if such a solution has enough regularity and geometrical properties to belong to one of those classes of Definition \ref{spazi}.

 When $\D=\emptyset$, namely in the free plate case, in Proposition \ref{GTxi} we provide the explicit representation of the gap function when the load is concentrated on the boundary. The same statement seems out of reach for more general load, that is for odd distributions such as
$\tfrac{\delta_{(z,w)}-\delta_{(z,-w)}}{2}$, with $z\in]0,\pi[$ and $w\in[0,\ell[$.
Nevertheless, it is reasonable to expect that the gap function amplifies whenever $w \to\ell$ and $z$ remains fixed. For this reason, we expect
$\overline T_{\pi/2}$ to be the worst case among all possible normalized couples of odd concentrated loads. This leads to the following.

\begin{conjecture}
When $D=\emptyset$, $\overline T_{\pi/2}$ and $-\overline T_{\pi/2}$ are the unique maximizers of the worst case problem \eqref{functional}.
\end{conjecture}

The worst case problem \eqref{functional} may also be set up in different (smaller) classes of loads such as $L^p$-spaces and one has the maximization problem \eqref{maxim}, see Theorem \ref{optimal}. We have no guess about the possible solutions of \eqref{maxim} when $p\in ]1,\infty[$. We also suspect that there
exists no maximizer for \eqref{maxim} in $L^1(\Omega)$; see Section~\ref{easy}. Moreover, we believe that the strange property stated in Theorem~\ref{fsymmetry2} for $p=\infty$ may not be fulfilled since we expect the following.

\begin{conjecture}
Let $p=\infty$. For every $D\subset\Omega$, the unique maximizers of the problem \eqref{maxim} are the odd function $f(x,y)=y/|y|$, $y\neq 0$, and its opposite $-f$.
\end{conjecture}

Finally, as concerns the most ambitious goal to solve the minimaxmax problem \eq{minimaxmax2}, we conclude by stating two conjectures which are supported by numerical computations. More precisely, Table \ref{tabella1} in Section~\ref{easy} suggests the following.
\begin{conjecture}
Let ${\mathcal F}$ and ${\mathcal D}$ be as in \eq{FDn}. The optimum of the minimaxmax problem \eq{minimaxmax2} is the couple $(f^1, D^0)$.
\end{conjecture}

Table \ref{table2} in Section~\ref{moregen} shows that the least $\GG^{\infty}_D$ is obtained for strips, then squares, hexagons, while the largest $\GG^{\infty}_D$ is obtained for triangles. This is somehow surprising since squares are expected to be in between triangles and hexagons. Moreover, Table \ref{table2} suggests the following.
\begin{conjecture}
Let $\mathcal F$ and ${\mathcal D}$ be as in \eq{FDn2}. The optimum of the minimaxmax problem \eq{minimaxmax2} is the couple ($\overline{e}_{1}, \text{Strips}$).
\end{conjecture}

\noindent
{\bf Acknowledgements.} The authors are grateful to J.B. Kennedy for his kind revision of the use of the English Language within the present paper.
The first, second, and fourth authors are partially supported by the Research Project FIR (Futuro in Ricerca) 2013 \emph{Geometrical and
qualitative aspects of PDE's}. The third author is partially supported by the PRIN project {\em Equazioni alle derivate parziali di tipo ellittico e parabolico:
aspetti geometrici, disuguaglianze collegate, e applicazioni}. The four authors are members of the Gruppo Nazionale per l'Analisi Matematica, la Probabilit\`a
e le loro Applicazioni (GNAMPA) of the Istituto Nazionale di Alta Matematica (INdAM).

\end{document}